\documentclass[11pt, reqno]{amsart}
\usepackage{amssymb}
\usepackage{times}
\usepackage{amsmath,amsthm}
\usepackage{amsfonts}
\usepackage{leftidx}
\usepackage{mathrsfs}
\usepackage{enumerate}
\usepackage{abstract}
\usepackage{stmaryrd}
\usepackage{ulem}

\usepackage{graphicx}
\usepackage{float}
\usepackage{hyperref}
\def\R{\mathbb{R}}

\def\d{|\nabla|}
\def\n{\nabla}

\def\p{\partial}

\def\be{\begin{equation}}
\def\ee{\end{equation}}
\def\vo{\vspace{1\baselineskip}}

\def\h{\frac{1}{2}}
\def\ta{\thickapprox}

\newtheorem{theorem}{Theorem}[section]

\newtheorem{lemma}[theorem]{Lemma}
\newtheorem{proposition}[theorem]{Proposition}
\theoremstyle{definition}
\newtheorem{definition}[theorem]{Definition}

\theoremstyle{remark}
\newtheorem{remark}[theorem]{Remark}

\numberwithin{equation}{section}
\begin{document}
 \title[3D Water Waves Above Flat Bottom]{On 3D water waves system above a flat bottom}
\author{Xuecheng Wang}
\address{Mathematics Department, Princeton University, Princeton, New Jersey,08544, USA}
\email{xuecheng@math.princeton.edu}
\thanks{}
\maketitle
\begin{abstract}
As a starting point of studying the long time behavior of the $3D$ water waves system in the flat bottom setting, in this paper,  we try to improve the understanding of the  Dirichlet-Neumann operator in this setting. As an application, we  study the $3D$ gravity   waves system and derive a new $L^2-L^\infty$ type energy estimate, which has a good structure in the $L^\infty-$type space. In our second paper \cite{wang2},  base  on the results we  obtained in this paper, we prove the  global regularity of the $3D$ gravity waves system for suitably small initial data.

\end{abstract}

\tableofcontents

\section{Introduction}
\subsection{The full water waves system above a flat bottom} 
We are interested in the long time behavior of the $3D$ water waves system  for suitably small initial data in the flat bottom  setting. 

The water waves system describes the evolution of an  inviscid incompressible fluid with constant density (e.g., water) inside a time dependent  water region   $\Omega(t)$, which has a free interface $\Gamma(t)$  and a fixed flat bottom $\Sigma$.  Above the domain $\Omega(t)$, it is vacuum.

Without loss of generality, we normalize  the depth of $\Omega(t)$ to be $1$. In the Eulerian coordinates system,   we can represent the domain ``$\Omega(t)$'', the interface ``$\Gamma(t)$'' and the bottom ``$\Sigma$''  as follows, 
\[
\Omega(t):= \{(x,y): x\in \R^2, -1\leq y\leq \,h(t,x)\},\]
\[ \Gamma(t):= \{(x,y): x\in \R^2, y= \,h(t,x)\}, \quad \Sigma :=\{(x,y): x\in \R^2, y=-1\}.
\]
We remark that, for the case we are considering, the size of $h(t,\cdot)$ will be small all   time.

We assume that the velocity field is  irrotational. The evolution of fluid is subjected to the gravity effect or the surface tension effect. We can describe the evolution of fluid   by the Euler equation as follows, 
\begin{equation}
\left\{\begin{array}{l}
\p_t u + u\cdot \nabla u =-\nabla p -g(0,0,1),\\
\\
\nabla\cdot u =0,\,\, \nabla \times u =0, \,\,u(0)=u_0,\\
\end{array}\right.
\end{equation}
where  ``$g$'' denotes the gravity effect constant.

Moreover, we have following boundary conditions,
\begin{equation}\label{boundaryconditions}
\left\{\begin{array}{lr}
u\cdot \vec{\mathbf{n}}=0 & \textup{on\,\,$\Sigma$},\\

P=\sigma H(h) & \textup{on\,\, $\Gamma(t),$}\\

\p_t + u\cdot\nabla \textup{tangents to $\cup_{t} \Gamma(t)$} & \textup{on\, $\Gamma(t)$,}
\end{array}\right.
\end{equation}
where $\sigma$ denotes the surface tension coefficient and ``$H(h)$ '' denotes the mean curvature of the interface, which is given as follows,
\[
H(h)= \nabla\cdot \Big(\frac{\nabla h}{\sqrt{1+|\nabla h|^2}}\Big).
\]

The first boundary condition in (\ref{boundaryconditions}) means that the fluid cannot go through the fixed bottom. The second boundary condition in (\ref{boundaryconditions})  comes from the Young-Laplace equation for the pressure. The third boundary condition in (\ref{boundaryconditions}) represents the kinematic boundary condition, which says that the free interface moves with the normal component of the velocity.

Recall that the velocity field is irrotational.  Hence, we can represent it in terms of a velocity potential $\phi$. We use $\psi$ to denote the restriction of the velocity potential on the boundary $\Gamma(t)$, i.e., $\psi(t,x):=\phi(t,x,\,h(t,x))$. From the  incompressible condition and boundary conditions, we can derive the following Laplace equation with two boundary conditions: Neumann type on the bottom and Dirichlet type on the interface,
\begin{equation}\label{harmoniceqn}
(\Delta_x + \p_y^2)\phi=0, \quad \frac{\p \phi}{\p \vec{\mathbf{n}}}\big|_{\Sigma}=0, \quad \phi\big|_{\Gamma(t)} = \psi.
\end{equation}

Hence, we can reduce (for example, see \cite{zakharov}) the motion of fluid inside the water region $\Omega(t)$ to the  evolution of height $\,h$ and the restricted velocity potential $\psi$ on the interface $\Gamma(t)$   as follows,
\begin{equation}\label{fullsystem}
\left\{\begin{array}{l}
\p_t \,h= G(\,h)\psi,\\
\\
\p_t \psi = - g h + \sigma H(h) - \frac{1}{2} |\nabla \psi|^2 + \displaystyle{\frac{(G(\,h)\psi + \n \,h\cdot\n \psi)^2}{2(1+ |\n\,h|^2)}},
\end{array}\right.
\end{equation}
where $G(\,h)\psi= \sqrt{1+|\n\,h|^2}\mathcal{N}(\,h)\psi$  and $\mathcal{N}(\,h)\psi$ is the Dirichlet-Neumann operator on the interface.

 The system (\ref{fullsystem}) has the following conservation law
\[
\mathcal{H}(\,h(t),\psi(t)) := \Big[ \int \frac{1}{2}\psi(t) G(\,h(t))\psi(t) +\frac{g}{2}  |\,h(t)|^2 + \frac{\sigma |\nabla h(t)|^2}{1+\sqrt{1+|\nabla h(t)|^2}}\Big]= \mathcal{H}(\,h(0), \psi(0)).
\] 

Intuitively speaking, after diagonalizing the system (\ref{fullsystem}), we find ourself dealing with the following type of  quasilinear dispersive  equation,
\begin{equation}\label{dispersivefullsystem}
(\p_t + i \widetilde{\Lambda}) u = \mathcal{N}(u, \nabla u),\quad \widetilde{\Lambda} = \sqrt{|\nabla| \tanh(|\nabla|)(g+\sigma |\nabla|^2)}, \quad u = \,h + i \widetilde{\Lambda}^{-1} |\nabla|\tanh{|\nabla|}\psi,
\ee
\be\label{diagonalizedvariable}
 u: \mathbb{R}_t \times \mathbb{R}_x^2\longrightarrow \mathbb{C}.
\end{equation}
Readers can temporarily take (\ref{dispersivefullsystem})   as granted. It would be much clearer after we obtain the linear term of the Dirichlet-Neumann operator, which is $\d\tanh\d \psi$, in section \ref{DNoperator}.

\subsection{Motivation and the main result of this paper.}
 Note that  the best decay rate that one can expect for a $2D$ dispersive equation is $1/t$, which is critical to establish the global regularity for small initial data.  

  To know the long time behavior,  it is crucial  to know what type of quadratic term we are dealing with.   Unfortunately, to the best of author's knowledge, there is no previous work that addresses this issue. It  motivates us to study this problem in this paper.

To identify the quadratic terms, it requires more careful analysis of the Dirichlet-Neumann operator in the flat bottom setting. Note that the water waves system  in the Eulerian coordinates  formulation (\ref{fullsystem}) is dimensionless.  Since we don't want to limit our scope to the $3$D setting, in this paper,  we will identify structures inside the Dirichlet-Neumann operator as many as possible.

To help readers understand what they will read about the  Dirichlet-Neumann operator in this paper, we   summarize and explain several important properties of  the  Dirichlet-Neumann operator here.  Those understandings will play   important roles in the long time behavior. They are listed as follows,

(i) Unlike the infinite depth setting,   in the flat bottom setting, we don't have   the null structure in the low frequency part.
  More precisely, if the frequencies of two inputs are $1$ and $0$ respectively, then   the size of symbol is ``$1$'' (flat bottom setting)  instead of ``$0$''(infinite depth setting).

 We remark that the principle symbol  of the  Dirichlet-Neumann operator  in the flat bottom setting is still same as in the infinite depth setting. Intuitively speaking, the high frequency parts of the  Dirichlet-Neumann operator in  two setting are almost same.

   (ii)   We give the explicit  formula of the quadratic terms of the Dirichlet-Neumann operator, which provides the first step to study the long time behavior of (\ref{dispersivefullsystem}).

(iii) We formulate the cubic and higher order terms of Dirichlet-Neumann operator   in a fixed point type formulation, which provides a good way to control the high order terms over time.

As a starting point and an example, we study  a specific  setting  of the water waves  system (\ref{fullsystem}), which is the gravity water waves system. More precisely, we consider the gravity effect and neglect the surface tension effect.  After normalizing the gravity effect constant ``$g$ ''to be ``$1$'',  the system (\ref{fullsystem}) is reduced as follows,
\begin{equation}\label{waterwave}
\left\{\begin{array}{l}
\p_t \,h= G(\,h)\psi,\\
 
\p_t \psi = -  h  - \frac{1}{2} |\nabla \psi|^2 + \displaystyle{\frac{(G(\,h)\psi + \n \,h\cdot\n \psi)^2}{2(1+ |\n\,h|^2)}}.
\end{array}\right.
\end{equation}

Correspondingly, the diagonalized equation  (\ref{dispersivefullsystem})    is reduced to the 
    quasilinear dispersive  equation as follows,
\begin{equation}\label{dispersive}
(\p_t + i \Lambda) u = \mathcal{N}(u, \nabla u),\quad \Lambda = \sqrt{|\nabla| \tanh(|\nabla|)}, \quad u = \,h + i\Lambda \psi.
\end{equation}

For the water waves system in the flat bottom setting, a typical issue is that the phases are  highly degenerated at the low frequency part. For example, we   consider a   phase associated with  a quadratic term  of (\ref{dispersive}) as follows,
\[
\Lambda(|\xi|) -\Lambda(|\xi-\eta|)+\Lambda(|\eta|) \approx [|\xi|-|\xi-\eta|+|\eta|] -\frac{1}{6}\big( |\xi|^3 -|\xi-\eta|^3 +|\eta|^3 \big), \quad    |\eta|\leq |\xi|\sim |\xi-\eta|\ll 1.
\]

When   ``$\xi$'' and ``$-\eta$'' are in the same direction, above phase is of size $|\xi|^2|\eta|$, which is highly degenerated.  Because of this issue, generally speaking, it is not hopeful to prove the sharp  $1/t$ decay rate of the nonlinear  solution over time. As a result, a rough energy estimate is not sufficient to control the growth of energy in the long run.

To get around this issue, instead of working too hard, it turns out that there is a relatively simple way to control the growth of energy. It relies on the   two observations for the system (\ref{waterwave}) as follows,

 (i) We can derive a new $L^2-L^\infty$ type   energy estimate    after carefully analyzing     the structures inside the quadratic terms  in (\ref{waterwave}). The input inside quadratic terms is not roughly putted in $L^\infty$ but in a weaker $L^\infty$ type space, which has derivatives in front.  See (\ref{energyestimate}).

  (ii)  The derivatives at low frequency part compensate  the decay rate of the solution of (\ref{waterwave}).  We can prove that the solution with some derivatives in front decays sharply, despite the fact that the solution itself may not have sharp decay rate. The proof of this fact involves a very delicate Fourier analysis. For those interested readers, please refer to \cite{wang2} for more details.

Before stating our main result, we define our  main function spaces as follows, 
\begin{equation}\label{equation1}
\| f \|_{\widetilde{W^\gamma}}:= \sum_{k \geq 0, k\in \mathbb{Z}}2^{\gamma k} \|P_{k} f\|_{L^\infty} + \| P_{\leq 0} f\|_{L^\infty},
\end{equation}
\begin{equation}\label{equation2}
\| f \|_{\widehat{W^{\gamma,\alpha}}}:= \sum_{ k\in \mathbb{Z}}(2^{\alpha k}+2^{\gamma k}) \|P_{k} f\|_{L^\infty}, \quad 0\leq \alpha \leq \gamma, \quad \| f\|_{\widehat{W^\gamma}}:= \| f \|_{\widehat{W^{\gamma,0}}}.
\end{equation}

\begin{theorem}\label{maintheorem}
Let $0< \delta < c$,    $  \alpha\in(0,1]$, and $N_0\geq 6$, where  $c$ is some sufficiently small constant.  If the initial data $(\,h_0, \Lambda \psi_0)\in H^{N_0+1/2}(\R^2)\times H^{N_0}(\R^2)$  satisfies the   smallness condition as follows,
\begin{equation}\label{smallnesscondition}
\| (h_0, \Lambda\psi_0)\|_{\widetilde{W^4}} \leq \delta, 
\end{equation}
then there exists $T > 0$ such that the system \textup{(\ref{waterwave})} has a unique solution $(\,h, \Lambda\psi)\in C^{0}\big([0,T]; H^{N_0}(\R^2)\times H^{N_0}(\R^2)\big)$. Moreover, we have a new type of energy estimate in the time interval of existence as follows,
\begin{equation}\label{energyestimate}
\frac{d}{d t}E_{N_0}(t)\lesssim_{N_0}  \big[ \|(\,h, \Lambda\psi)(t)\|_{\widehat{W^{4,\alpha}}} + \| (\,h, \Lambda\psi)(t)\|_{\widehat{W^4}}^2\big] E_{N_0}(t),
\end{equation}
where the energy $E_{N_0}(t)$ is defined in \textup{(\ref{energyfunction})}. The size of energy is comparable to $\|(\,h, \Lambda\psi)(t)\|_{H^{N_0}}^2$.
\end{theorem}

\begin{remark}
Note that smallness condition is not assumed  in  \cite{alazard1, alazard2,alazard3, lannes}  to derive the local wellposedness. For the 
purpose of obtaining global solution, we impose the  smallness condition (\ref{smallnesscondition})   to derive (\ref{energyestimate}) as our main goal, which is the first step to obtain global existence for small initial data. 
\end{remark}

In our second paper \cite{wang2}, base  on the results we obtained in this paper, we show  that the solution of the system (\ref{waterwave}) globally exists and scatters to a linear solution. We will study the long time behaviors of the water waves system (\ref{fullsystem}) in other settings  in the future. For example, the capillary waves system and the gravity-capillary waves system. We expect that the results we obtained in this paper  will be very helpful to the future study of water waves system in the flat bottom setting. 

\subsection{Previous results.}  To be concise,  we mainly discuss works on the local behavior of the water waves system in this subsection.   For more detailed discussion on the long time behavior, please refer to the introduction  of our second paper \cite{wang2} and references therein.

 Starting from the work of Nalimov \cite{nalimov} and Yoshihara \cite{yosihara}, there are considerable amount of works on the local theory of the water waves system. In the framework of Sobolev spaces and without smallness assumptions on the initial data, the local wellposedness was first obtained by Wu \cite{wu1, wu2} for the gravity waves system.  The local wellposedness was also obtained when the surface tension effect is   effective  by Beyer-G\"unther  \cite{beyer}. Later, different methods were developed and many important results were obtained to improve our understanding on the local behavior of the water waves system.  Among many of them, we mention the following works:  Christodoulou-Lindblad \cite{christodoulou}, Ambrose-Masmoudi \cite{ambrose}, Lannes \cite{lannes}, Shatah-Zeng \cite{shatah1}, Coutand-Shkoller \cite{coutand1}, Alazard-Burq-Zuily \cite{alazard1, alazard2, alazard3}. Interested readers may refer there and the references therein for more details.

Roughly speaking,  the local existence of the water waves system (\ref{fullsystem}) holds even when the initial interface  has a unbounded curvature and  the   bottom   is very rough. A fixed length separation between the interface and the bottom is sufficient.  See the works of Alazard-Burq-Zuily \cite{alazard1, alazard2, alazard3} and Lannes \cite{lannes} for more details and more precise descriptions.

 \subsection{Main ideas and the outline of this paper}
To prove our main theorem, we have to pay attentions to both the low frequency part and the high frequency part. 

For the high frequency part, due to the quasilinear nature of the gravity waves system (\ref{waterwave}), we have to get around the difficulty of losing  one derivative. Thanks to the work of Lannes \cite{lannes}, the work of Alazard-M\'etivier \cite{alazard4} and the works of  Alazard-Burq-Zuily \cite{alazard1, alazard2, alazard3}, we can utilize the  paralinearization method to get around the potential loss of one derivative. However, for their purposes, only the high frequency part has been paid attentions in their works.  In this paper, we will  do the paralinearization process   and pay special attentions to the low frequency part at the same time. 

For the low frequency part, more  careful estimates of the Dirichlet-Neumann operator are necessary and essential. Since it is not straightforward to see the fact  that we can gain ``$\alpha$ ''derivatives for the input that is putted in $L^\infty-$ type space. For example, for the quadratic term   $\nabla h \cdot \nabla \psi$ of the Dirichlet-Neumann operator, it is problematic to gain $\alpha$ derivatives  when $\psi$ has smaller frequency   because   the total number of derivatives of $\psi$ in (\ref{energyestimate})   is  $1+\alpha$  at the low frequency part  when the input $\psi$ of the quadratic terms is putted in $L^\infty$.

 To close the argument,  we will use the hidden structures inside the system (\ref{waterwave}) for different scenarios.  Without involving too many details, we give two examples as follows to explain the main ideas behind. (i) When $\psi$ has a smaller frequency inside $\nabla h\cdot \nabla \psi$, we can use the hidden symmetry  to move one derivative from $\nabla h$ to  $\nabla \psi$ during energy estimate, hence we have two derivatives in total for $\psi$. (ii) For some terms, e.g., the good remainder term of  the paralinearization  process,  we can lower their regularities to $L^2$. Hence, we can  put $\nabla\psi$ in $L^2$  and  put $\nabla\,h$ in $L^\infty$, as a result the desire estimate (\ref{energyestimate}) also  holds for this case.

\vo

\noindent \textbf{Outline:\quad }In section \ref{preliminary}, we introduce   notations and give a quick summary of the paradifferential calculus. In section \ref{DNoperator}, we study various properties of the Dirichlet-Neumann operator. In section \ref{parasymm}, we use the paralinearization method to show the good structures inside the system (\ref{waterwave}), which   help us to find   good substitution   variables. In section \ref{energyestimatesection}, we  prove the new energy estimate (\ref{energyestimate}) by using the symmetries inside the  equations satisfied by the  good substitution variables. In the appendix, we   calculate explicitly the quadratic terms of good remainder terms. It aims to help readers   understand the fact that we can gain  gain ``$\alpha$ ''derivatives in (\ref{energyestimate}) for the inputs  of quadratic terms, which are putted	 in the $L^\infty-$ type space.

\vo
\noindent \textbf{Acknowledgement} \quad I thank my Ph.D. advisor Alexandru Ionescu for many helpful discussions and suggestions. The first version of manuscript was done when I was visiting Fudan University and BICMR, Peking University. I thank Prof. Zhen Lei and BICMR for their warm hospitalities during the visits.
\section{Preliminary}\label{preliminary}

\subsection{Notations} For any two numbers $A$ and $B$, we use  $A\lesssim B$ and $B\gtrsim A$ to denote  $A\leq C B$, where $C$ is an absolute constant. We use   $A\lesssim_{\epsilon} B$ to denote $A\leq C_{\epsilon} B$, where  constant $C_{\epsilon}$ depends on $\epsilon$. For an integer $k\in\mathbb{Z}$, we use $k_{+}$ to denote $\max\{k,0\}$ and   use $k_{-}$ to denote $\min\{k,0\}$.

Throughout this paper, we will abuse the notation of $``\Lambda"$. When  there is no lower script associated with $\Lambda$, then  $\Lambda:=\sqrt{\tanh(|\nabla|)|\nabla|}$, which is the linear operator associated with the  system (\ref{dispersive}). For $p\in \mathbb{N}_{+}$, we  use   $\Lambda_{p}(\mathcal{N})$ to denote the $p$-th order terms of a nonlinearity $\mathcal{N}$  when a Taylor expansion  for the nonlinearity $\mathcal{N}$ is available.   For example, $\Lambda_{2}[\mathcal{N}]$ denotes the quadratic term of $\mathcal{N}$.   We also use   $\Lambda_{\geq p}[\mathcal{N}]$ to denote the $p$-th  and higher orders terms. More precisely, $\Lambda_{\geq p}[\mathcal{N}]:=\sum_{q\geq p}\Lambda_{q}[\mathcal{N}]$.  In this paper, 
 the Taylor expansion and $\Lambda_{p}[\cdot]$  are in terms of $\,``h"$ and $``\psi"$ when there is no special annotation.

We fix an even smooth function $\tilde{\psi}:\R \rightarrow [0,1]$, which is supported in $[-3/2,3/2]$ and equals to $1$ in $[-5/4, 5/4]$. For any $k\in \mathbb{Z}$, define
\[
\psi_{k}(x) := \tilde{\psi}(x/2^k) -\tilde{\psi}(x/2^{k-1}), \quad \psi_{\leq k}(x):= \tilde{\psi}(x/2^k), \quad \psi_{\geq k}(x):= 1-\psi_{\leq k-1}(x).
\]
Denote the projection operators $P_{k}$, $P_{\leq k}$ and $P_{\geq k}$ by the Fourier multipliers $\psi_{k},$ $\psi_{\leq k}$ and $\psi_{\geq k }$ respectively. For a well defined function $f$, we will also use the notation $f_{k}$ to abbreviate $P_{k} f$.

 The Fourier transform is defined as follows, 
\[
\mathcal{F}(f)(\xi) = \int_{\R^2} e^{-i x\cdot \xi} f (x) d x. 
\] 
 For two well defined functions $f$ and $g$ and  a bilinear form  $Q(f,g)$, we will use the convention that the symbol $q(\cdot, \cdot)$ of $Q(\cdot, \cdot)$  is defined in the following sense throughout this paper,
\begin{equation}
\mathcal{F}[Q(f,g)](\xi)= \frac{1}{4\pi^2} \int_{\R^2} \widehat{f}(\xi-\eta)\widehat{g}(\eta)q(\xi-\eta, \eta) d \eta. 
\end{equation}
Meanwhile, for a trilinear form $C(f, g, h)$, its symbol $c(\cdot, \cdot, \cdot)$ is defined in the following sense, 
\[
\mathcal{F}[C(f,g,h)](\xi) =  \frac{1}{16\pi^4} \int_{\R^2}\int_{\R^2} \widehat{f}(\xi-\eta)\widehat{g}(\eta-\sigma) \widehat{h}(\sigma) c(\xi-\eta, \eta-\sigma, \sigma) d \eta d \sigma.
\]

\subsection{Multilinear estimate}

We define a class of symbol and its associated norms as follows,
\[
\mathcal{S}^\infty:=\{ m: \mathbb{R}^4\,\textup{or}\, \mathbb{R}^6 \rightarrow \mathbb{C}, m\,\textup{is continuous and }  \quad \| \mathcal{F}^{-1}(m)\|_{L^1} < \infty\},
\]
\[
\| m\|_{\mathcal{S}^\infty}:=\|\mathcal{F}^{-1}(m)\|_{L^1}, \quad \|m(\xi,\eta)\|_{\mathcal{S}^\infty_{k,k_1,k_2}}:=\|m(\xi, \eta)\psi_k(\xi)\psi_{k_1}(\xi-\eta)\psi_{k_2}(\eta)\|_{\mathcal{S}^\infty},
\]
\[
 \|m(\xi,\eta,\sigma)\|_{\mathcal{S}^\infty_{k,k_1,k_2,k_3}}:=\|m(\xi, \eta,\sigma)\psi_k(\xi)\psi_{k_1}(\xi-\eta)\psi_{k_2}(\eta-\sigma)\psi_{k_3}(\sigma)\|_{\mathcal{S}^\infty}.
\]

\begin{lemma}\label{multilinearestimate}
Assume that $m$, $m'\in S^\infty$, $p, q, r, s \in[1, \infty]$ , then the following estimates hold for  well defined functions $f(x), g(x)$, and $h(x)$, 
\begin{equation}\label{productofsymbol}
\| m\cdot m'\|_{S^\infty} \lesssim \| m \|_{S^\infty}\| m'\|_{S^\infty},
\end{equation}
\begin{equation}\label{bilinearesetimate}
\Big\| \mathcal{F}^{-1}\big[\int_{\R^2} m(\xi, \eta) \widehat{f}(\xi-\eta) \widehat{g}(\eta) d \eta\big]\Big\|_{L^p} \lesssim \| m\|_{\mathcal{S}^\infty}\| f \|_{L^q}\| g \|_{L^r}  \quad \textup{if}\,\,\, \frac{1}{p} = \frac{1}{q} + \frac{1}{r},
\end{equation}
\begin{equation}\label{trilinearesetimate}
\Big\| \mathcal{F}^{-1}\big[\int_{\R^2}\int_{\R^2} m'(\xi, \eta,\sigma) \widehat{f}(\xi-\eta) \widehat{h}(\sigma) \widehat{g}(\eta-\sigma)  d \eta d\sigma\big] \Big\|_{L^{p}} \lesssim \|m'\|_{\mathcal{S}^\infty} \| f \|_{L^q}\| g \|_{L^r} \| h\|_{L^s},\,\, \end{equation}
where $ \displaystyle{\frac{1}{p} = \frac{1}{q} + \frac{1}{r} + \frac{1}{s}}.$
\end{lemma}

To estimate the $\mathcal{S}^{\infty}_{k,k_1,k_2}$ norm  and the $\mathcal{S}^{\infty}_{k,k_1,k_2,k_3}$ norm  of symbols, we   constantly use the following lemma . 
\begin{lemma}\label{Snorm}
For $i\in\{1,2,3\}, $ if $f:\mathbb{R}^{2i}\rightarrow \mathbb{C}$ is a smooth function and $k_1,\cdots, k_i\in\mathbb{Z}$, then the following estimate holds,
\begin{equation}\label{eqn61001}
\| \int_{\mathbb{R}^{2i}} f(\xi_1,\cdots, \xi_i) \prod_{j=1}^{i} e^{i x_j\cdot \xi_j} \psi_{k_j}(\xi_j) d \xi_1\cdots  d\xi_i \|_{L^1_{x_1, \cdots, x_i}} \lesssim \sum_{m=0}^{i+1}\sum_{j=1}^i 2^{m k_j}\|\p_{\xi_j}^m f\|_{L^\infty} .
 \end{equation}
\end{lemma}
\begin{proof}
The case when $i=1,3$ can be estimated in the same way as the case when $i=2$. We only do the case $i=2$ in details here. Through scaling, it is  sufficient to prove above estimate for the case when $k_1=k_2=0$. From Plancherel theorem, we have the following two estimates, 
\[
 \| \int_{\mathbb{R}^{2i}} f(\xi_1,\xi_2) e^{i (x_1\cdot \xi_1+ x_2\cdot \xi_2)} \psi_{0}(\xi_1) \psi_{0}(\xi_2) d \xi_1 d\xi_2 \|_{L^2_{x_1, x_2}}\lesssim \| f(\xi_1, \xi_2)\|_{L^\infty_{\xi_1, \xi_2}},
\]
\[
 \| (|x_1|+|x_2|)^3 \int_{\mathbb{R}^{2i}} f(\xi_1,\xi_2) e^{i (x_1\cdot \xi_1+ x_2\cdot \xi_2)} \psi_{0}(\xi_1) \psi_{0}(\xi_2) d \xi_1 d\xi_2 \|_{L^2_{x_1, x_2}}\lesssim  \sum_{m=0}^3\big[\|\p_{\xi_1}^m f\|_{L^\infty} + \| \p_{\xi_2}^m f \|_{L^\infty}\big],
\]
which are sufficient to finish the proof of (\ref{eqn61001}).
\end{proof}

\subsection{Paradifferential calculus}
In this subsection, we discuss some necessary background materials of the paradifferential calculus. For more details and related topics, please refer to \cite{para} and references therein.
\begin{definition}
Given $\rho\in \mathbb{N}_{+}, \rho \geq 0$ and $m\in \R$, we use $\Gamma^{m}_{\rho}(\R^2)$  to denote the space of locally bounded functions $a(x,\xi)$ on $\R^2\times (\R^2/\{0\})$, which are $C^{\infty}$ with respect to $\xi$ for $\xi\neq 0 $. Moreover, they satisfy the following estimate, 
\[
\forall |\xi|\geq 1/2, \| \p_{\xi}^{\alpha} a(\cdot, \xi)\|_{W^{\rho, \infty}}\lesssim_{\alpha} (1+|\xi|)^{m-|\alpha|}, \quad \alpha\in \mathbb{N}^2,
\]
where $W^{\rho, \infty}$ is the usual Sobolev space. Note that $W^{\rho, \infty}$ contains the spaces $\widetilde{W^\rho}$ and $\widehat{W^{\rho, \alpha}}$, which are defined in (\ref{equation1}) and (\ref{equation2}),  as subspaces.
\end{definition}

\begin{remark}
$\rho$ in above definitions is not necessary an integer, but the integer case is sufficient for our purpose.
\end{remark}

\begin{definition}
\begin{enumerate}
\item[(i)] We use   $\dot{\Gamma}^{m}_{\rho}(\R^2)$ to denote the subspace of $\Gamma^{m}_{\rho}(\R^2)$, which consists of symbols that are homogeneous of degree $m$ in $\xi.$ \\
\item[(ii)] If $a= \displaystyle{\sum_{0\leq j < \rho} a^{(m-j)}}$, where $a^{(m-j)}\in \dot{\Gamma}^{m-j}_{\rho-j}(\R^2)$, then we say $a^{(m)}$ is the principal symbol of $a$.\\
\item[(iii)] An operator $T$ is said to be of order $m$, $m\in \R$, if for all $\mu\in\R$, it's bounded from $H^{\mu}(\R^2)$ to $H^{\mu-m}(\R^2)$. We use $S^{m}$ to denote the set of all operators of order m.
\end{enumerate}
\end{definition}
For symbol $a\in \Gamma^{m}_{\rho}$, we can define its norm as follows,
\[
M^{m}_{\rho}(a):= \sup_{|\alpha|\leq 2+\rho} \sup_{|\xi|\geq 1/2} \| (1+|\xi|)^{|\alpha|-m}\p_{\xi}^{\alpha} a (\cdot, \xi)\|_{W^{\rho, \infty}}.
\]

For  
$a, f\in L^2$ and  a pseudo differential operator $\tilde{a}(x,\xi)$, we define the operator $T_{a} f$ and $T_{\tilde{a}} f$ as follows,
 \begin{equation}\label{eqn1001}
T_a f = \mathcal{F}^{-1}[\int_{\R} \widehat{a}(\xi-\eta) \theta(\xi-\eta, \eta)\widehat{f}(\eta) d \eta]
,\,\, T_{\tilde{a}} f = \mathcal{F}^{-1} [ \int_{\R} \mathcal{F}_x(\tilde{a})(\xi-\eta,\eta) \theta(\xi-\eta,\eta)\widehat{f}(\eta) d \eta  ],
\end{equation}
where the cut-off function  $\theta(\xi-\eta, \eta) $ is defined as follows,
\[
\theta(\xi-\eta, \eta) = \left\{\begin{array}{ll}
1 & \textup{when}\,\,|\xi-\eta|\leq 2^{-10} |\eta|, |\eta| \geq 1,\\
0 & \textup{when}\,\, |\xi-\eta| \geq 2^{10} |\eta|\,\, \textup{or $|\eta|\leq 1$}.\\
\end{array}\right.
\]
For two well defined functions $a$ and $b$, we have the   paraproduct decomposition as follows,
\begin{equation}\label{equation340}
a b = T_a b + T_b a + \mathcal{R}(a,b),
\end{equation}
where $\mathcal{R}(a,b)$ contains those terms in which $a$ and $b$ have comparable size of frequencies or the frequency of output is less than ``$1$''.

 We have the following composition lemma for paradifferential operators. It can be found, for example, in \cite{ alazard1,para}.
\begin{lemma}\label{composi}
Let $m\in \R$ and $\rho >0$. If given symbols $a\in \Gamma_{\rho}^{m}(\R^d)$ and $b\in\Gamma_{\rho}^{m'}(\R^d)$,  we   define
\[
a\sharp b = \sum_{|\alpha|< \rho} \frac{1}{i^{|\alpha|} \alpha!} \p_{\xi}^{\alpha} a \p_{x}^{\alpha}b,
\]
then for all $\mu\in\R$, there exists a constant $K$ such that 
\begin{equation}\label{eqn700}
\| T_a T_b - T_{a\sharp b}\|_{H^{\mu}\rightarrow H^{\mu-m-m'+\rho}} \leq K M^{m}_{\rho}(a) M^{m'}_{\rho}(b).
\end{equation}
\end{lemma}
\begin{remark}
It may be too early to give this remark here, but we think that it is a good idea to  keep the following simple observation in mind,   which will be very helpful to see the equivalence relations later. The simple observation is that if the  symbols $a$ and $b$ are all depends on $\nabla\,h$ instead of $\,h$, then the rough estimate (\ref{eqn700}) is sufficient to gain one derivative at low frequency part. 

\end{remark}

\begin{lemma}\label{adjoint}
 Let $m\in \R$,  $\rho >0$ and  $a\in \Gamma^{m}_{\rho}(\R^d)$. If we use $(T_{a})^{\ast}$ to denote the adjoint operator of $T_a$ and use $\bar{
a}$ to denote   the complex conjugate of $a$, then  $(T_{a})^{\ast}- T_{a^{\ast}}$ is of order $m-\rho$, where
\[
a^{\ast} = \sum_{|\alpha|< \rho} \frac{1}{i^{|\alpha|} \alpha !} \p_{\xi}^{\alpha} \p_{x}^{\alpha} \bar{a}.
\]
Moreover, the operator norm of $(T_{a})^{\ast} - T_{a^{\ast}}$ is bounded by $M^{m}_{\rho}(a).$
\end{lemma}
\begin{proof}
See \cite{alazard1}[Theorem 3.10].
\end{proof}
\begin{remark}
In most  applications of Lemma \ref{adjoint}, we have $m\leq 1$. If we let $\rho=1$ in above lemma, then it's easy to see $a^{\ast}= \bar{a}$. If moreover  $a$ is real, then $a^{\ast}=\bar{a}=a$. 
\end{remark}

\section{Dirichlet-Neumann operator}\label{DNoperator}

The main goal of this section is to study various properties of the Dirichlet-Neumann operator, which provide a foundation to do the process of paralinearization and symmetrization in section \ref{parasymm} and obtain the new energy estimate (\ref{energyestimate}) in section \ref{energyestimatesection}.  The study of the Dirichlet-Neumann operator is mainly reduced to study the velocity potential inside the water region $\Omega(t)$.

% We first give  a quick summary about what we will mainly do in this section. (i) We will study the velocity potential in two types of formulations. We will use one of them to  estimate  the Dirichlet-Neumann operator and use the    other one to do the  paralinearization process for the Dirichlet-Neumann operator. (ii) We will give   the explicit formula of the quadratic terms of the Dirichlet-Neumann operator and  its associated symbol. (iii) We will  formulate the cubic and higher order terms of the Dirichlet-Neumann operator in a fixed point type formulation, which provides a good way to estimate them in the small data regime. 

%The goal of doing (ii) and (iii) is to identify more structures of the  Dirichlet-Neumann operator. However, they are not so relevant to the proof of the new energy estimate (\ref{energyestimate}). Readers can safely skip them, i.e., subsections \ref{secondorderexpansion} and \ref{fixedpointcubic}, if not interested.

Recall the smallness condition (\ref{smallnesscondition}) of the initial data.  From the local wellposedness result of the gravity waves system (\ref{waterwave}), we know that there exists a   positive time $T$, such that the following estimate holds, 
\begin{equation}\label{smallnessestimate}
\sup_{t\in[0,T]} \| (h, \Lambda\psi)(t)\|_{\widetilde{W^4}}\leq 2\delta,
\end{equation}
which means that the $L^\infty$ norm of solution remains small in the time interval $[0,T].$ Throughout the rest of this paper, we restrict ourself to the time interval $[0,T].$

\subsection{Type I formulation of the Laplace equation (\ref{harmoniceqn})}

In this subsection, we   reduce the Laplace equation  (\ref{harmoniceqn})  into a favorable formulation such that we can solve it and identify the fixed point type structure inside the Laplace equation, which further enables us to estimate the Dirichlet-Neumann operator.

We   do change of variables and map the water region $\Omega(t)$ to the strip $\mathcal{S}:=\R^2\times[-1,0]$ as follows,
\[
(x,y)\rightarrow (x,z), \quad z: = \displaystyle{\frac{y-h(t,x)}{h(t,x)+1}}.
\]
Very naturally, the inverse of transformation is given as follows, 
\[
y= h + (h+1)z.
\]

Define the velocity potential in the $(x,z)$ coordinate system as $\varphi(x,z):= \phi(x, h+ (h+1)z)$. From direct computations, we have    the following identities, 
\begin{equation}\label{eqn640}
\phi(x,y)= \varphi(x, \frac{y-h}{1+h}), \quad \p_y \phi = \frac{\p_z\varphi }{1+h},\quad \p_y^2 \phi = \frac{\p_z^2 \varphi}{(1+h)^2},
\end{equation}
\begin{equation}\label{eqn641}
\p_{x_i}  \phi = \p_{x_i} \varphi + \p_z \varphi \big[\frac{-\p_{x_i}h}{1+h}- \frac{(y-h)\p_{x_i} h}{(1+h)^2}\big] = \p_{x_i}\varphi - \frac{(y+1)\p_{x_i}h}{(1+h)^2}\p_z\varphi,
\end{equation}
\[\p_{x_i}^2 \phi = \p_{x_i}^2 \varphi - 2 \frac{(y+1)\p_{x_i} h}{(1+ h)^2}\p_z\p_{x_i}\varphi + [\frac{-(y+1)\p_{x_i}^2 h}{(1+ h)^2} + 2\frac{(y+1)(\p_{x_i} h)^2}{(1+ h)^3}] \p_z\varphi + \frac{(y+1)^2(\p_{x_i} h)^2}{(1+h)^4} \p_z^2\varphi.
\]
From above identities and (\ref{harmoniceqn}), it's easy to derive the following equation,
\begin{equation}\label{changeco1}
(\Delta_{x}+\p_y^2)\phi = 0 \Longrightarrow P_{x,z}\varphi:=[ \Delta_{x}+ \tilde{a}\p_z^2 +  \tilde{b}\cdot \nabla \p_z + \tilde{c}\p_z ] \varphi=0, 
\end{equation}
where
\begin{equation}\label{coeff}
\tilde{a}= \frac{(y+1)^2|\nabla h|^2}{(1+h)^4}  + \frac{1}{(1+h)^2}=\frac{1+(z+1)^2|\nabla h|^2}{(1+h)^2},
\end{equation}
\begin{equation}\label{coeff1}
\tilde{b}=- 2 \frac{(y+1)\nabla h}{(1+h)^2} = \frac{-2(z+1)\nabla h}{1+h},\quad \tilde{c}= \frac{-(z+1)\Delta_{x} h}{(1+h)} + 2\frac{(z+1)|\nabla h|^2}{(1+h)^2}. 
\end{equation}

To sum up, we can reduce the Laplace equation (\ref{harmoniceqn}) with two boundary conditions    in terms of  ``$\varphi$'' as follows,
\begin{equation}\label{harmonic}
P_{x,z} \varphi =0, \quad \varphi\big|_{z=0}=\psi, \quad \p_z\varphi\big|_{z=-1} = 0, \quad (x,z)\in \R^2\times [-1,0].
\end{equation}

\subsection{Type II formulation of the Laplace equation (\ref{harmoniceqn}) }\label{type2}
In this subsection, we reduce the Laplace equation (\ref{harmoniceqn}) into another favorable formulation, which will be used  to do the paralinearization of the Dirichlet-Neumann operator in subsection \ref{paralinearDN}.  

We remark that we don't use the type I formulation (\ref{harmonic}) to do the paralinearization process because the coefficients $\tilde{a}, \tilde{b}, \tilde{c}$ in (\ref{coeff}) and (\ref{coeff1})  are very complicated, which complicate the paralinearization process and prevent us to see clearly the principle symbol of the Dirichlet-Neumann operator.

Recall the smallness condition (\ref{smallnessestimate}). Since  the height of interface  is very small, we know that  there exists a curve parallel to the interface $\Gamma(t)$ with depth  $1/2$ inside $\Omega(t)$. More precisely, we have
\[
\Omega_1(t):=\{(x,y): x\in \R^2, \quad \,h(t,x)-1/2\leq y \leq \,h(t,x)\}, \quad \Omega_1(t) \subset \Omega(t).
\]
  Define 
 \begin{equation}\label{region2}
\Omega_2(t):=\{(x,y): x\in \R^2, \quad \,h(t,x)-1/4\leq y \leq \,h(t,x)\}, \quad \Omega_2(t)\subset \Omega_1(t) \subset \Omega(t),
\end{equation}
\be\label{ee1}
\tilde{\phi}(x,y):= \chi(y-\,h(t,x)) \phi(x,y), \quad (x,y)\in \Omega_1(t),
 \quad \chi(z)=1\, \textup{if $z\geq -1/4$}, \textit{supp}(\chi)\subset [-1/2,0],\ee
where $\chi(x)$  is a fixed Schwarz function. 

Recall the Laplace equation (\ref{harmoniceqn}). From   (\ref{ee1}), it is easy to derive the  following identities
\begin{equation}\label{cutoffharmonic}
\Delta_{x,y}\tilde{\phi} = \tilde{g}:= \Delta_{x,y}[\chi \phi ]- \chi \Delta_{x,y}\phi, \quad (x,y)\in \Omega_1(t),
\end{equation}
\[
\tilde{\phi}(x,y)= {\phi}(x,y),\quad \tilde{g}(x,y)=0, \quad (x,y)\in \Omega_2(t).
\]

We can map the water region $\Omega_1(t)$ to the strip $\mathcal{S}':=\R^2 \times [-1/2,0]$ by changing the coordinate system as follows, 
\[
(x, y)\rightarrow (x, w),\quad w:= y-\,h(t,x).
\]
Define the velocity potential in the $(x,w)$ coordinate system as  $\Phi(x,w):= \tilde{\phi}(x, \omega+\,h(t,x)).$ Hence $\tilde{\phi}(x,y)= \Phi(x,y -\,h(t,x))$. From (\ref{cutoffharmonic}), it's easy to verify that the following equality holds, 
\begin{equation}\label{changco2}
\mathcal{P}_{x,w} \Phi:=[ \Delta_{x}+ a'\p_w^2 +  b'\cdot \nabla \p_w + c'\p_w ]\Phi = g'(x,w):= \tilde{g}(x,\omega+ h(t,x)), 
\end{equation}
where
\begin{equation}\label{coeff2}
a'= 1 + |\nabla \,h|^2, \quad b'= -2\nabla \,h, \quad c'= -\Delta \,h. 
\end{equation}

\begin{remark}
From  (\ref{coeff}),  (\ref{coeff1}), and (\ref{coeff2}), it is easy to see that the coefficients in the equation (\ref{changco2}) satisfied by ``$\Phi$''  are much easier and more favorable than the coefficients in  the equation (\ref{changeco1}) satisfied by ``$\varphi$''.  However, the equation (\ref{changco2})  satisfied by $\Phi$ cannot be used as the starting point because we don't know the   estimates of $\Phi$ at the first place. 
\end{remark}

From above definitions,   the following identities hold inside the water region $\Omega_2(t)$ (see (\ref{region2})) and its corresponding regions in the new coordinate systems,
\[
\Phi(x,w)= \varphi(x,\frac{w}{1+\,h}), \quad \varphi(x,z)= \Phi(x,(1+\,h)z),\quad  (x,w)\in \R^2\times[-1/4, 0],
\]
\be\label{eqq4}
\p_{x_i}\Phi = \p_{x_i}\varphi -\frac{w \p_z\varphi \p_{x_i} h }{(1+h)^2}
\quad \p_{w} \Phi = \frac{\p_{z}\varphi}{1+\,h}.
\ee

From (\ref{eqn640}) and (\ref{eqq4}), the Dirichlet-Neumann operator $G( h)\psi $ in terms of ``$\varphi$'' and ``$\Phi$''  and  the quadratic term of $G(\,h)\psi$  are given as follows,
\begin{equation}\label{DN1}
G( h)\psi = [-\nabla\,h\cdot\nabla \phi + \p_y\phi]\big|_{y=\,h} = \frac{1+|\nabla \,h|^2}{1+\,h} \p_z \varphi \big|_{z=0} -\nabla\psi \cdot \nabla \,h,\quad 
\end{equation}
\begin{equation}\label{DN2}
G(h)\psi=(1+|\nabla \,h|^2) \p_{w}\Phi\big|_{w=0} - \nabla\,h \cdot \nabla\psi,
\end{equation}
\begin{equation}\label{quadraticformula}
\Lambda_{2}[G(\,h)\psi]= \Lambda_{2}[\p_z\varphi\big|_{z=0}] - \Lambda_1[\p_z\varphi\big|_{z=0}]\,h - \nabla\psi \cdot \nabla \,h. 
\end{equation}

\subsection{A fixed point type formulation for the Dirichlet-Neumann operator} In this subsection, our main goal is   to obtain basic estimates of the Dirichlet-Neumann operator with special attention to the low frequency part, which will further help us to obtain a new energy estimate. 

To this end, we  study the reduced Laplace equation (\ref{harmonic}) and formulate $\nabla_{x,z}\varphi$ into a fixed point type formulation, which enables us to use a fixed point type argument. 

We first rewrite the equation (\ref{harmonic}). After moving all nonlinear terms to the right hand, we can rewrite the  equation (\ref{harmonic}) as follows,
\be\label{eqq40}
\p_z^2\varphi + \Delta_x\varphi = (\p_z -\d)(\p_z +\d) \varphi= g(z) := (1-\tilde{a})\p_z^2 \varphi - \tilde{b}\cdot \nabla \p_z \varphi - \tilde{c}\p_z\varphi.
\ee

Now, we will solve $\varphi(z)$ from (\ref{eqq40}) by treating ``$g(z)$''  in (\ref{eqq40})  as a given nonlinearity.  Define $\tilde{h}(x,z):= (\p_z - \d)\varphi$. Very naturally, we have
\begin{equation}
\left\{\begin{array}{l}
(\p_z + \d)\tilde{h} = g,\\
\\
(\p_z- \d)\varphi = \tilde{h}, \quad \varphi\big|_{z=0}=\psi,\p_{z}\varphi\big|_{z=-1}=0.\\
\end{array}\right.
\end{equation}
We can solve above system of equations with $\tilde{h}(-1)$ to be determined as follows,
\[
\tilde{h}(z) = e^{-z\d}\tilde{h}(-1) + \int_{-1}^{z} e^{-(z-z')\d} g(z') d z',\quad 
\]
\[
\varphi(z) = e^{z\d} \varphi(0) + \int_{0}^{z} e^{(z-z')\d} \tilde{h}(z') d z'
= e^{z\d}\psi - \int_{z}^{0} e^{(z-z')\d} [ e^{-z'\d} \tilde{h}(-1)\]
\[ + \int_{-1}^{z'} e^{-(z'-s)\d} g(s) d s] d z'
=e^{z\d}\psi -\h \d^{-1}[e^{-z\d}- e^{z\d} ]\tilde{h}(-1)\]
\[- \int_{-1}^{z} \int_{z}^{0} e^{(z+s-2z')\d} g(s) d z' d s - \int_{z}^{0} \int_{s}^0 e^{(z+s-2z')\d} g(s) d z' d s
\]
\begin{equation}\label{eqn10}
= e^{z\d}\psi -\h\d^{-1}[e^{-z\d}- e^{z\d} ]\tilde{h}(-1) +\h \int_{-1}^{0} \d^{-1}e^{(z+s)\d} g(s) d s- \h \int_{-1}^{0} \d^{-1} e^{-|z-s|\d} g(s) d s. 
\end{equation}

The unknown $\tilde{h}(-1)$ is determined by the Neumann type boundary condition $\p_{z}\varphi|_{z=-1}=0$. We calculate $\p_z \varphi $ from the formula (\ref{eqn10}) and have the following equality,
\[
\p_z \varphi = \d e^{z\d}\psi  + \h [e^{z\d} + e^{-z\d}]\tilde{h}(-1) + \h \int_{-1}^{0} e^{(z+s)\d} g(s) ds  - \h \int_{-1}^{0} e^{-|z-s|\d} \textup{sign}(s-z) g(s) d s.
\]
After evaluating above equality at point $z=-1$, we have
\begin{equation}
\tilde{h}(-1)= - \frac{2\d e^{-\d}\psi}{e^{-\d} + e^{\d}} -\int_{-1}^0 \frac{e^{(s-1)\d} - e^{-(s+1)\d}}{e^{-\d} + e^{\d}} g(s)\, d s ,
\end{equation}
which further gives us 
\begin{equation}
\p_z \varphi =  \frac{e^{(z+1)\d}- e^{-(z+1)\d}}{e^{-\d}+e^{\d}} \d \psi -\h \frac{e^{z\d}+e^{-z\d}}{e^{-\d} +e^{\d}} \int_{-1}^0 [e^{(s-1)\d}- e^{-(s+1)\d}] g(s) \, d s 
\end{equation}
\begin{equation}
+ \h \int_{-1}^{0} e^{(z+s)\d} g(s) ds  - \h \int_{-1}^{0} e^{-|z-s|\d} \textup{sign}(s-z) g(s) d s.
\end{equation}
Moreover, we can reduce (\ref{eqn10})    as follows, 
\[
\varphi(z) = \Big[ \frac{e^{-(z+1)\d}+ e^{(z+1)\d}}{e^{-\d} + e^{\d}}\Big]\psi + \h\d^{-1}\frac{e^{-z\d}- e^{z\d} }{e^{-\d} + e^{\d}}\int_{-1}^0 \big[e^{(s-1)\d} - e^{-(s+1)\d}\big] g(s)\, d s
\]
\begin{equation}\label{eqn11}
+\h \int_{-1}^{0} \d^{-1}e^{(z+s)\d} g(s) d s - \h \int_{-1}^{0} \d^{-1} e^{-|z-s|\d} g(s) d s. 
\end{equation}

However, we can not use the formulation  (\ref{eqn11}) to estimate the velocity potential and the Dirichlet-Neumann operator because $g(z)$ actually depends on the  velocity potential $\varphi(z)$, see (\ref{eqq40}). 

To get around this issue. we observe that there exists a    fixed point type structure inside $g(z)$. Recall (\ref{eqq40}), (\ref{coeff}), and (\ref{coeff1}). Note that
\[
g= \p_z [ \frac{2\,h+\,h^2 - (z+1)^2 |\nabla\,h|^2}{(1+\,h)^2} \p_z \varphi +\frac{2(z+1)\nabla\,h\cdot \nabla\varphi}{1+\,h}] - \frac{2\nabla \,h \cdot \nabla\varphi}{1+\,h}  + \frac{(z+1)\Delta \,h}{1+\,h}\p_z \varphi,
\]
and
\[
 \frac{(z+1)\Delta \,h}{1+\,h}\p_z \varphi = \nabla \cdot [\frac{(z+1)\nabla \,h \p_z\varphi}{1+\,h}] + \frac{(z+1)|\nabla \,h|^2 \p_z\varphi}{(1+\,h)^2} -\p_z[ \frac{(z+1)\nabla \,h \cdot \nabla \varphi}{1+\,h} ] + \frac{\nabla \,h\cdot \nabla \varphi}{1+\,h}.
\]
Hence, we can decompose the nonlinearity $g(z)$ into three parts as follows, 
\be\label{eqq20}
g(z)= \p_z g_1(z)+g_2(z) + \nabla\cdot g_3(z),
\ee
 where
\begin{equation}\label{eqn12}
g_1(z) =  \frac{2\,h+\,h^2 - (z+1)^2 |\nabla\,h|^2}{(1+\,h)^2} \p_z \varphi +\frac{(z+1)\nabla\,h\cdot \nabla\varphi}{1+\,h},\quad g_1(-1)=0,
\end{equation}
\begin{equation}\label{eqn14}
g_2(z) =\frac{(z+1)|\nabla \,h|^2 \p_z\varphi}{(1+\,h)^2}  - \frac{\nabla \,h \cdot \nabla\varphi}{1+\,h} ,\quad g_3(z)= \frac{(z+1)\nabla \,h \p_z\varphi}{1+\,h}.
\end{equation}
To simplify the notation, we define 
\begin{equation}\label{tildeeta}
 \tilde{\,h}_1:= \frac{2\,h+\,h^2}{(1+\,h)^2}, \quad    \tilde{h}_2:=\frac{|\nabla \,h|^2}{(1+\,h)^2}, \quad \tilde{h}_3:=  \frac{\nabla \,h}{1+\,h}.
 \ee
As a result, we have
\begin{equation}\label{eqn782}
g_1(z)=\tilde{\,h}_1 \p_z \varphi - (z+1)^2 \tilde{h}_2  \p_z\varphi + (z+1) \tilde{h}_3\cdot \nabla \varphi,
\end{equation}
\begin{equation}\label{eqn783}
g_2(z)=  (z+1) \tilde{h}_2\p_z \varphi - \tilde{h}_3\cdot \nabla \varphi, \quad g_3(z)=(z+1)   \tilde{h}_3 \p_z \varphi.
\end{equation}
Note that $g_1(z)$, $g_2(z)$, and $g_3(z)$ are all linear with respect to $\nabla_{x,z}\varphi(z)$.

    After decomposing  $g(s)$ in (\ref{eqn11}) into three parts: ``$\p_s g_1 $'', ``$g_2$'' and ``$\nabla\cdot g_3$'', we  do integration by parts in ``$s$'' to move the derivative ``$\p_s$'' in front of   ``$\p_s g_1$'' around.  As a result,   the following equality holds, 
\[
\varphi(z)=\Big[ \frac{e^{-(z+1)\d}+ e^{(z+1)\d}}{e^{-\d} + e^{\d}}\Big]\psi + \h\d^{-1}\frac{e^{-z\d}- e^{z\d} }{e^{-\d} + e^{\d}}\int_{-1}^0 \big[e^{(s-1)\d}  (g_2 +\nabla\cdot g_3 -\d g_1)  \]
\[- e^{-(s+1)\d}(g_2 + \nabla \cdot g_3+\d g_1)  \big] d s +\h \int_{-1}^{0} \d^{-1}e^{(z+s)\d} [g_2 +\nabla\cdot g_3 -\d g_1] d s
\]
\begin{equation}\label{solution}
  - \h \int_{-1}^{0} \d^{-1} e^{-|z-s|\d} [g_2 +\nabla\cdot g_3 -\textup{sign}(z-s)\d g_1] d s. 
\end{equation}
Now, we know that the nonlinearity in (\ref{solution}) is linear with respect to $\nabla_{x,z}\varphi$. 

To see the fixed point type structure of $\nabla_{x,z}\varphi$, we take the derivative  ``$\nabla_{x,z}$'' on the both hand  sides of (\ref{solution}). As a result,  we   derive the  fixed point type formulation for $\nabla_{x,z}\varphi$ as follows, 
\[
\nabla_{x,z}\varphi = \Bigg[ \Big[ \frac{e^{-(z+1)\d}+ e^{(z+1)\d}}{e^{-\d} + e^{\d}}\Big]\nabla\psi ,  \frac{e^{(z+1)\d}- e^{-(z+1)\d}}{e^{-\d}+e^{\d}} \d \psi\Bigg] + \]
\[
\Bigg[\h \frac{\nabla}{\d}\frac{e^{-z\d}- e^{z\d} }{e^{-\d} + e^{\d}}\int_{-1}^0  \big[e^{(s-1)\d}  (g_2 + \nabla \cdot g_3  -\d g_1)  - e^{-(s+1)\d}(g_2 + \nabla \cdot g_3+\d g_1) \big] d s,\]
\[ -\h \frac{e^{z\d}+e^{-z\d}}{e^{-\d} +e^{\d}} \int_{-1}^0  \big[e^{(s-1)\d}  (g_2+\nabla \cdot g_3 -\d g_1)  - e^{-(s+1)\d}(g_2 +\nabla\cdot g_3 +\d g_1) \big] d s 
  \Bigg] + 
\]
\[
\Bigg[ \h \int_{-1}^{0} \frac{\nabla}{\d}e^{(z+s)\d} [g_2+ \nabla \cdot g_3-\d g_1] d s - \h \int_{-1}^{0} \frac{\nabla}{\d} e^{-|z-s|\d} [g_2+ \nabla \cdot g_3- \textup{sign}(z-s)\d g_1] d s,\]
\begin{equation}\label{fixedpoint2}
\h \int_{-1}^{0} e^{(z+s)\d} [g_2+ \nabla \cdot g_3-\d g_1] ds  - \h \int_{-1}^{0} e^{-|z-s|\d}  [\textup{sign}(s-z)(g_2+ \nabla \cdot g_3) + \d g_1]d s\Bigg]+ [\mathbf{0}, g_1(z)].\end{equation}
To simplify the notation, we define  operators as follows,
\begin{equation}\label{equation300}
K_1(z,s):=\Big[ \frac{\nabla}{2\d}\frac{e^{-z\d}- e^{z\d} }{e^{-\d} + e^{\d}}e^{(s-1)\d} +\frac{\nabla}{2\d}e^{(z+s)\d} ,  -\h \frac{e^{z\d}+e^{-z\d}}{e^{-\d} +e^{\d}}e^{(s-1)\d} + \frac{1}{2}e^{(z+s)\d}  \Big],
\end{equation}
\begin{equation}\label{equation301}
K_2(z,s):= \Big[ \frac{\nabla}{2\d}\frac{e^{-z\d}- e^{z\d} }{e^{-\d} + e^{\d}}e^{-(s+1)\d}\,\, , \,\,
  -\h \frac{e^{z\d}+e^{-z\d}}{e^{-\d} +e^{\d}} e^{-(s+1)\d} \Big],
\end{equation}
\begin{equation}\label{equation302}
K_3(z,s)= \Big[  \frac{\nabla}{2\d}e^{-|z-s|\d} \,\, , \,\, \frac{1}{2}e^{-|z-s|\d}\textup{sign($s-z$)}\Big].
\end{equation}
With above operators, we can rewrite (\ref{fixedpoint2}) as follows,
\[\nabla_{x,z}\varphi = \Bigg[ \Big[ \frac{e^{-(z+1)\d}+ e^{(z+1)\d}}{e^{-\d} + e^{\d}}\Big]\nabla\psi ,  \frac{e^{(z+1)\d}- e^{-(z+1)\d}}{e^{-\d}+e^{\d}} \d \psi\Bigg] + [\mathbf{0}, g_1(z)]+ \]
\[
+\int_{-1}^{0} [K_1(z,s)-K_2(z,s)-K_3(z,s)](g_2(s)+\nabla \cdot g_3(s))  ds \]
\begin{equation}\label{fixedpoint}
+\int_{-1}^{0} K_3(z,s)\d\textup{sign($z-s$)}g_1(s)  -\d [K_1(z,s) +K_2(z,s)]g_1(s)\, d  s.
\end{equation}
  
To make sure that we can close the fixed point type argument, we need to check the boundednesses of operators $K_i(z,s)$ such that the issue of losing derivative doesn't exist. More precisely,   the following Lemma holds.

\begin{lemma}\label{integraloperator}
For $k, \gamma \geq 0$, we have the following estimates, 
\begin{equation}\label{eqn2300}
\sum_{i=1,2,3}\| \int_{-1}^{0} K_i(z,s)\nabla g(s) ds \|_{L^\infty_z H^k} + \| \int_{-1}^{0} K_i(z,s) g(s) ds\|_{L^\infty_z H^k} \lesssim \|g(z)\|_{L^\infty_z H^k},
\end{equation}
\[
\sum_{i=1,2,3}\| \int_{-1}^{0} K_i(z,s)\nabla g(s) ds\|_{L^\infty_z \widetilde{W^\gamma}}
\]
\begin{equation}\label{eqn2301}
 + \| \int_{-1}^{0} [K_1(z,s)-K_2(z,s)-K_3(z,s)] g(s) ds \|_{L^\infty_z\widetilde{W^\gamma}} \lesssim \| g(z)\|_{L^\infty_z \widetilde{W^\gamma}}.
\end{equation}
\end{lemma}
\begin{proof}
We first prove the desired estimate (\ref{eqn2300}).
Recall (\ref{equation300}), (\ref{equation301}), and (\ref{equation302}). From   Lemma \ref{Snorm},   the following estimates hold,	
\begin{equation}\label{equation330}
\sup_{z,s\in[-1,0]}\|\mathcal{F}^{-1}[\mathcal{F}\big([K_1(z,s)-K_2(z,s)-K_3(z,s)-[0,(1-\textup{sign($s-z$)})/2]]\big)(\xi)\psi_{k_1}(\xi)] \|_{L^1} \lesssim 2^{k_{1,-}},  
\end{equation}
\begin{equation}\label{equation320}
\sup_{z,s\in[-1,0]}\sum_{i=1,2,3}\|\mathcal{F}^{-1}[\mathcal{F}\big(K_i(z,s)\big)(\xi)\psi_{k_1}(\xi)] \|_{L^1} \lesssim 1. 
\end{equation}

We will use above estimates for the case when $k_1< 0$. However, when $k_1\geq 0$, we can not use the estimate  (\ref{equation320}) directly to estimate the left hand side of (\ref{eqn2300}), otherwise	 we lose one derivative. An  important observation is that the integration with respect to ``$s$'' actually compensates the loss.

 For any fixed $k\geq 0, k\in \mathbb{Z}$, we have the following formulation in terms of kernel,
\begin{equation}\label{equation331}
\int_{-1}^{0} K_i(z,s)\nabla P_k[g(s)] ds  = \int_{-1}^0\int_{\R^2} K_{i;k}(z,s,y) g(s,x-y) d y d s, \quad 
\end{equation}
where
\[
K_{i;k}(z,s,y) = \int_{\R^2} e^{i y \cdot \xi} \mathcal{F}(K_1(z,s))(\xi)\psi_k(\xi) \xi d \xi.
\]
After integration by parts in $\xi$ many times, the following point-wise estimate holds for $i\in\{1,2,3\}$,
\begin{equation}\label{equation332}
|K_{i;k}(z,s,y)| \lesssim 2^{3k} \big( 1 +2^k |y| + 2^k|z-s| \big)^{-10},
\end{equation}
which further implies that the kernel ``$K_{i;k}(z,s,y)$'' belongs to $L^1_{s,y}$ for fixed $z$.
Therefore, from (\ref{equation320}) and (\ref{equation332}), the following estimate holds, 
\[
\big|\textup{The left hand side of (\ref{eqn2300})}\big|^2 \lesssim \sum_{k_1\leq 0} \| P_{k_1}[g(z)]\|_{L^\infty_z L^2}^2 + \sum_{i=1,2,3} \sum_{k_1\geq 0 }  2^{2kk_1} \| K_{i;k_1}(z,s)\|_{L^1_{s,y}}^2\]
\[\times \| P_{k_1}[g(z)]\|_{L^\infty_z L^2}^2   \lesssim \| g(z)\|_{L^\infty_z H^k}^2. 
\]
 Hence finishing the proof of (\ref{eqn2300}).  Very similarly, from (\ref{equation330}), (\ref{equation320}), and (\ref{equation332}), our desired estimate (\ref{eqn2301}) follows in the same way.  
\end{proof}

From (\ref{fixedpoint}) and estimates in Lemma \ref{integraloperator}, now, it is clear that we can estimate $\nabla_{x,z}\varphi$ by using a fixed point type argument.  

However, if we do it roughly, then the resulted estimate will not tell the difference between $\nabla_x \varphi$ and $\p_z \varphi$. To  capture the fact that $\p_z \varphi$ actually has two derivatives at the low frequency part while $\nabla_x \varphi$ only has one derivative, we 	decompose  $\nabla_{x,z}\varphi$ as follows, 
\be\label{eqq51}
\nabla_{x,z}\varphi = \Lambda_{1}[\nabla_{x,z}\varphi] + \Lambda_{\geq 2}[\nabla_{x,z}\varphi].
\ee
From (\ref{fixedpoint}), it's easy to see that  $ \Lambda_{1}[\nabla_{x,z}\varphi]$ is given as follows, 
 \begin{equation}\label{eqn773}
 \Lambda_{1}[\nabla_{x,z}\varphi]= \Bigg[ \Big[ \frac{e^{-(z+1)\d}+ e^{(z+1)\d}}{e^{-\d} + e^{\d}}\Big]\nabla\psi ,  \frac{e^{(z+1)\d}- e^{-(z+1)\d}}{e^{-\d}+e^{\d}} \d \psi\Bigg].
 \end{equation}
From (\ref{eqn773}), it is easy to see that $\Lambda_{1}[\p_z \varphi]$ has two derivatives at the low frequency part. Now, the goal  is reduced to estimate  $\Lambda_{\geq 2}[\nabla_{x,z}\varphi]$, which is done again by a fixed-point type argument.

Recall (\ref{fixedpoint}). To identify the   fixed point type structure inside $\Lambda_{\geq 2}[\nabla_{x,z}\varphi]$, it is sufficient to reformulate $\Lambda_{\geq 2}[g_i(z)], i \in \{1,2,3\}$.

 Recall (\ref{eqn782}) and (\ref{eqn783}). After using the decomposition (\ref{eqq51})  for $\nabla_{x,z}\varphi$ in  $g_{i}(z), i \in \{1,2,3\}$, we have the decomposition of  $\Lambda_{\geq 2}[g_i(z)], i \in \{1,2,3\}$, as follows,
\begin{equation}\label{eqn2001}
\Lambda_{\geq 2}[g_1(z)] = \tilde{\,h}_1 \Lambda_{\geq 2}[\p_z \varphi] - (z+1)^2  \tilde{\,h}_2 \Lambda_{\geq 2}[\p_z\varphi] + (z+1) \tilde{\,h}_3\cdot\Lambda_{\geq 2}[ \nabla \varphi]
\end{equation}
\begin{equation}\label{eqn2002}
+ \tilde{\,h}_1 \Lambda_{1}[\p_z \varphi] - (z+1)^2  \tilde{\,h}_2 \Lambda_{1}[\p_z\varphi] + (z+1)  \tilde{\,h}_3\cdot \Lambda_{1}[ \nabla \varphi],
\end{equation}
\begin{equation}\label{eqn2003}
\Lambda_{\geq 2}[g_2(z)]= (z+1)  \tilde{\,h}_2\Lambda_{\geq 2}[\p_z \varphi] -   \tilde{\,h}_3\cdot \Lambda_{\geq 2}[ \nabla \varphi]+  (z+1)   \tilde{\,h}_2\Lambda_{1}[\p_z \varphi] -  \tilde{\,h}_3\cdot \Lambda_{1}[\nabla \varphi],
\end{equation}
\begin{equation}\label{eqn2004}
\Lambda_{\geq 2}[g_3(z)]=(z+1)  \tilde{\,h}_3 \Lambda_{\geq 2}[\p_z \varphi] + (z+1)  \tilde{\,h}_3 \Lambda_{1}[\p_z \varphi].
\end{equation}
 
From (\ref{eqn2002}), (\ref{eqn2003}), and (\ref{eqn2004}), now it is easy to see that there exists the fixed point type structure for $\Lambda_{\geq 2}[\nabla_{x, z}\varphi]$ in $\Lambda_{\geq 2}[g_i(z)], i \in \{1,2,3\}$. From the standard fixed point type argument and the estimates in Lemma \ref{integraloperator}, we obtain basic estimates for $\Lambda_{\geq 2}[\nabla_{x, z}\varphi]$, which further gives us a more precise estimate for $ \nabla_{x, z}\varphi $ from (\ref{eqq51}).

More precisely, our main results in this subsection is summarized as follows,
\begin{lemma}\label{Sobolevestimate}
For $\gamma', k' \geq 1$,  $0< \delta \ll 1 $, $\alpha\in(0,1],$ if $\,h \in \widetilde{W^{\gamma'}}\cap H^{k'}$ satisfies the following smallness assumption:
\begin{equation}\label{smallness}
\| \,h\|_{\widetilde{W^{\gamma'}}} < \delta,
 \end{equation}
  then  the following $L^2-$ type estimate and $L^\infty-$ type estimate of the velocity potential ``$\varphi$'' hold,
\begin{equation}\label{eqn2200}
 \|\nabla_{x,z}\varphi\|_{L^\infty_z H^k}  \lesssim \| \nabla\psi\|_{H^k} + \| \,h\|_{H^{k+1}} \| \nabla\psi\|_{\widetilde{W^0}},\quad 
\end{equation}
\begin{equation}\label{eqn2201}
\| \nabla_{x}\varphi\|_{L^\infty_z \widetilde{W^\gamma}} \lesssim \| \nabla \psi \|_{\widetilde{W^{\gamma}}} , \quad \| \p_z \varphi\|_{L^\infty_z \widetilde{W^{\gamma}}}\lesssim\| \nabla \psi\|_{\widehat{W}^{\gamma,\alpha}}+  \| \,h\|_{\widetilde{W^{\gamma+1}}} \| \nabla \psi \|_{\widetilde{W^\gamma}},
\end{equation}
\begin{equation}\label{eqn2203}
\|\Lambda_{\geq 2}[\nabla_{x,z}\varphi]\|_{L^\infty_z\widetilde{W^{\gamma}}} \lesssim \| \nabla\psi\|_{\widetilde{W^\gamma}}\| \,h\|_{\widetilde{W^{\gamma+1}}},
\end{equation}
\begin{equation}\label{eqn2202}
\| \Lambda_{\geq 2}[\nabla_{x,z}\varphi]\|_{L^\infty_z H^k} \lesssim  \| \,h\|_{\widetilde{W^1}}\| \d\psi\|_{H^k} + \| \nabla \psi\|_{\widetilde{W^0}}\| \,h\|_{H^{k+1}},
\end{equation}
where $k\leq k'-1$ and $1\leq \gamma \leq \gamma'-1$. In above estimates, the range of $z$ for the $L^\infty_z$ norm is  $[-1,0].$
\end{lemma}

\begin{proof}
We first estimate $\Lambda_{\geq 2}[\nabla_{x,z}\varphi]$. Recall (\ref{fixedpoint}), (\ref{eqn2001}), (\ref{eqn2002}), (\ref{eqn2003}) and (\ref{eqn2004}). From estimate   (\ref{eqn2301}) in Lemma \ref{integraloperator},  the following estimate holds   
\[
\| \Lambda_{\geq 2}[\nabla_{x,z}\varphi]\|_{L^\infty_z \widetilde{W^\gamma}} \lesssim \|\Lambda_{\geq 2}[(g_1(z), g_2(z), g_3(z))] \|_{L^\infty_z \widetilde{W^\gamma}}\]
\[
\lesssim \| \,h\|_{\widetilde{W^{\gamma+1}}} \| \Lambda_{\geq 2}[\nabla_{x,z}\varphi]\|_{L^\infty_z \widetilde{W^\gamma}} + \|\,h\|_{\widetilde{W^{\gamma+1}}} \|\nabla\psi\|_{\widetilde{W^\gamma}}. 
\]
Hence, from the smallness condition (\ref{smallness}), we have
\begin{equation}\label{Linfty}
\| \Lambda_{\geq 2}[\nabla_{x,z}\varphi]\|_{L^\infty_z \widetilde{W^\gamma}} \lesssim \| \,h\|_{\widetilde{W^{\gamma+1}}} \| \nabla\psi\|_{\widetilde{W^\gamma}}.
\end{equation}
Very similarly, from estimate (\ref{eqn2300}) in Lemma \ref{integraloperator}, the following estimate holds,
\[
\| \Lambda_{\geq 2}[\nabla_{x,z}\varphi]\|_{L^\infty_z H^k} \lesssim \|\Lambda_{\geq 2}[(g_1(z), g_2(z), g_3(z))] \|_{L^\infty_z H^k}\lesssim \| \,h\|_{\widetilde{W^1}} \| \Lambda_{\geq 2}[\nabla_{x,z}\varphi]\|_{L^\infty_z H^k}\]
\[
 + \| \,h\|_{H^{k+1}} \| \Lambda_{\geq 2}[\nabla_{x,z}\varphi]\|_{L^\infty_z \widetilde{W^0}} + \| \,h\|_{H^{k+1}}\| \nabla\psi\|_{\widetilde{W^0}} + \| \nabla \psi\|_{H^{k}} \| \,h\|_{\widetilde{W^1}}
\]
\[
\lesssim \| \,h\|_{\widetilde{W^1}} \| \Lambda_{\geq 2}[\nabla_{x,z}\varphi]\|_{L^\infty_z H^k} 
+ \| \,h\|_{H^{k+1}}\|\nabla\psi\|_{\widetilde{W^0}}(1+ \| \,h\|_{\widetilde{W^1}}) + \| \nabla \psi\|_{H^{k}} \| \,h\|_{\widetilde{W^1}}.
\]
Again, from the smallness assumption (\ref{smallness}), we conclude
\begin{equation}\label{L2type}
\| \Lambda_{\geq 2}[\nabla_{x,z}\varphi]\|_{L^\infty_z H^k} \lesssim \| \,h\|_{H^{k+1}}\|\nabla\psi\|_{\widetilde{W^0}} + \| \nabla \psi\|_{H^{k}} \| \,h\|_{\widetilde{W^1}}.
\end{equation}
From estimates (\ref{Linfty}) and (\ref{L2type}) and the explicit formulas of $\Lambda_{1}[\nabla_{x,z}\varphi]$ in (\ref{eqn773}), we have
\[
\| \nabla_{x,z}\varphi\|_{L^\infty H^k} \lesssim \| \nabla\psi\|_{H^k} +  \| \,h\|_{H^{k+1}}\|\nabla\psi\|_{\widetilde{W^0}},\quad\|\nabla_{x}\varphi\|_{L^\infty_z \widetilde{W^\gamma}} \lesssim \|\nabla\psi\|_{\widetilde{W^\gamma}},
\]
\[
\| \p_z \varphi\|_{L^\infty_z \widetilde{W^\gamma}} \lesssim \| \Lambda^2\psi\|_{\widetilde{W^\gamma}} + \| \,h\|_{\widetilde{W^{\gamma+1}}} \| \nabla\psi\|_{\widetilde{W^\gamma}}\lesssim \|\nabla\psi\|_{\widehat{W}^{\gamma, \alpha}} + \| \,h\|_{\widetilde{W^{\gamma+1}}} \| \nabla\psi\|_{\widetilde{W^\gamma}}.
\]
Hence  finishing the proof.
\end{proof}

\subsection{ The quadratic terms  of the Dirichlet-Neumann operator }\label{secondorderexpansion} The content of this subsection is not related to the proof of our main theorem. However, it is crucial to the study of the long time behavior of the water waves system in the flat bottom.

 Generally speaking, the main enemies of the global existence of a $2D$ dispersive equation are the quadratic terms. The  first step is to know exactly what the enemies are. Surprisingly, as a byproduct of  the fixed point type formulation  (\ref{fixedpoint}),  we can calculate explicitly the quadratic terms of the Dirichlet-Neumann operator.

More precisely, the main result of this subsection is stated as follows,
\begin{lemma}\label{quadraticterms}
In terms of $h$ and $\psi$, the quadratic terms of   the Dirichlet-Neumann operator is given as follows,
\be\label{eqq345}
\Lambda_{2}[G(h)\psi] = -\nabla\cdot( h \nabla\psi) - \d \tanh\d( h \d \tanh\d \psi).
\ee
\end{lemma}

\begin{remark}
Before we proceed to   prove above Lemma, we compare the main difference between the flat bottom setting, which is less studied, and the infinite depth setting, which is well studied recently. In the infinite depth setting, the quadratic terms of the Dirichlet-Neumann operator is given as follows,
\be\label{eqq259}
\textup{(Infinite depth setting)}\qquad \qquad  \Lambda_{2}[G(h)\psi] = -\nabla\cdot( h \nabla\psi) - \d  ( h \d   \psi).
\ee

If  the frequency  $ \eta$   of ``$\psi$'' is of size $1$ and the frequency  $ \xi-\eta$ of ``$h$''   is of size $0$, from (\ref{eqq345}) and (\ref{eqq259}),   
it is easy to check the size of the symbol of quadratic terms as follows,
\[
\textup{symbol of quadratic term in the flat bottom:} \quad \xi\cdot \eta - |\xi||\eta| \tanh|\xi| \tanh|\eta| = \frac{4|\xi|^2}{(e^{|\xi|}+e^{-|\xi|})^2}\sim 1,
\]
\[
\textup{symbol of quadratic term in the infinite depth setting:}\qquad  -|\xi||\eta| + \xi \cdot \eta =0.
\]
 That is to say, unlike the infinite depth setting, we don't have null structure at the low frequency part in  the flat bottom setting. As a result,  we expect much stronger nonlinear effect from the quadratic terms, which makes the global regularity problem in the flat bottom setting more delicate and more difficult than the infinite depth setting.

 \end{remark}

\noindent \textit{Proof of Lemma} \ref{quadraticterms}.\qquad Recall (\ref{DN1}) and (\ref{eqn773}). We have
\be\label{eqqn1}
\Lambda_{2}[G(h)\psi]= \Lambda_{2}[\p_z\varphi\big|_{z=0}] - h \d \tanh\d \psi - \nabla h \cdot \nabla \psi.
\ee
Hence,  the problem is reduced to calculate explicitly the quadratic terms of $\p_z\varphi\big|_{z=0}$. Recall    (\ref{fixedpoint}) , we have 
\[
\Lambda_{2}[\p_z \varphi \big|_{z=0}] = -\frac{1}{e^{\d}+e^{-\d}} \int_{-1}^0 [e^{(s-1)\d}-e^{-(s+1)\d}] [\Lambda_{2}[g_2+\nabla\cdot g_3]] ds 
\]
\[
+\frac{1}{e^{\d}+e^{-\d}} \int_{-1}^0 [e^{(s-1)\d}+e^{-(s+1)\d}] \d [\Lambda_{2}[g_1]] ds + \int_{-1}^{0}e^{s\d} \Lambda_{2}[g_2+\nabla\cdot g_3-\d g_1] d s + \Lambda_{2}[g_1(0)]
\]
\begin{equation}\label{eqn40}
= \int_{-1}^0 \frac{e^{(s+1)\d} + e^{-(s+1)\d}}{e^{\d}+ e^{-\d}}\Lambda_{2}[g_2+\nabla\cdot g_3] d s - \int_{-1}^{0} \frac{e^{(s+1)\d} -e^{-(s+1)\d}}{e^{\d}+e^{-\d}} \d\Lambda_{2}[g_1(s)] d s   + \Lambda_{2}[g_1(0)].
\end{equation}
 From  (\ref{eqn12}),  (\ref{eqn14}), and (\ref{eqn773}), it is easy to derive the following equalities, 
\[
\Lambda_{2}[g_1(s)]= 2\,h \frac{e^{(s+1)\d} -e^{-(s+1)\d}}{e^{\d}+e^{-\d}}\d\psi + (s+1)\nabla \,h \cdot \nabla \frac{e^{(s+1)\d} + e^{-(s+1)\d}}{e^{\d}+e^{-\d}}\psi,
\]
\begin{equation}\label{eqn2008}
\Lambda_{2}[g_1(0)] = 2  h \d\tanh\d  \psi + \nabla \,h\cdot \nabla \psi,
\end{equation}
\[
\Lambda_{2}[g_2(s)]=  - \nabla \,h\cdot \nabla \frac{e^{(s+1)\d} + e^{-(s+1)\d}}{e^{\d}+e^{-\d}}\psi,\quad \Lambda_{2}[g_3(s)]=(s+1)\nabla \,h \frac{e^{(s+1)\d} -e^{-(s+1)\d}}{e^{\d}+e^{-\d}}\d\psi.
\]
After plugging above explicit formulas of $\Lambda_2[g_i(z)]$, $i\in\{1,2,3\}$, the goal is reduced to calculate explicitly the symbols of two integrals in (\ref{eqn40}).  Define
\be\label{eqnn2}
Q_1(\,h,\psi):=
\int_{-1}^{0} \frac{e^{(s+1)\d} + e^{-(s+1)\d}}{e^{\d}+ e^{-\d}}\Lambda_{2}[g_2+\nabla\cdot g_3] d s = Q_{1,1}(\,h,\psi)+ Q_{1,2}(\,h, \psi), 
\ee
\be\label{eqnn6}
Q_2(\,h, \psi) = - \int_{-1}^{0} \frac{e^{(s+1)\d} -e^{-(s+1)\d}}{e^{\d}+e^{-\d}} \d\Lambda_{2}[g_1] d s= Q_{2,1}(\,h, \psi) + Q_{2,2}(\,h,\psi),
\ee
where
\[
Q_{1,1}(\,h, \psi) = \int_{-1}^{0} \frac{e^{(s+1)\d} + e^{-(s+1)\d}}{e^{\d}+ e^{-\d}}\Big[
\nabla \cdot[(s+1)\nabla \,h \frac{e^{(s+1)\d} -e^{-(s+1)\d}}{e^{\d}+e^{-\d}}\d\psi \Big] d s,
\]
\[
Q_{1,2}(\,h, \psi) = \int_{-1}^{0} \frac{e^{(s+1)\d} + e^{-(s+1)\d}}{e^{\d}+ e^{-\d}}\Big[
- \nabla \,h\cdot \nabla \frac{e^{(s+1)\d} + e^{-(s+1)\d}}{e^{\d}+e^{-\d}}\psi \Big] d s,
\]
\[
Q_{2,1}(\,h,\psi)= - \int_{-1}^{0} \frac{e^{(s+1)\d} -e^{-(s+1)\d}}{e^{\d}+e^{-\d}} \d\Big[2\,h \frac{e^{(s+1)\d} -e^{-(s+1)\d}}{e^{\d}+e^{-\d}}\d\psi\Big] d  s,
\]
\[
Q_{2,2}(\,h, \psi)= - \int_{-1}^{0} \frac{e^{(s+1)\d} -e^{-(s+1)\d}}{e^{\d}+e^{-\d}} \d\Big[(s+1)\nabla \,h \cdot \nabla \frac{e^{(s+1)\d} + e^{-(s+1)\d}}{e^{\d}+e^{-\d}}\psi\Big] d  s.
\]
The   symbol $q_{i,j}(\xi-\eta, \eta)$ of bilinear operator $Q_{i,j}(h, \psi)$, $i,j\in\{1,2\},$ is given as follows, 
\[
q_{1,1}(\xi-\eta,\eta)= \frac{-\xi\cdot (\xi-\eta)|\eta|}{(e^{|\xi|}+e^{-|\xi|})(e^{|\eta|}+e^{-|\eta|})} \int_{-1}^{0} (s+1)[e^{(s+1)|\xi|} + e^{-(s+1)|\xi|}][e^{(s+1)|\eta|} - e^{-(s+1)|\eta|}]  d s
\] 
\[
= \frac{-\xi\cdot (\xi-\eta)|\eta|}{(e^{|\xi|}+e^{-|\xi|})(e^{|\eta|}+e^{-|\eta|})}\Big[\frac{(|\xi|+|\eta|-1)e^{|\xi|+|\eta|}-(-|\xi|-|\eta|-1)e^{-|\xi|-|\eta|}}{(|\xi|+|\eta|)^2}
\]
\begin{equation}\label{eqn42}
+ \frac{(|\eta|-|\xi|-1)e^{|\eta|-|\xi|}-(|\xi|-|\eta|-1)e^{|\xi|-|\eta|}}{(|\xi|-|\eta|)^2}\Big],
\end{equation}

\[
q_{1,2}(\xi-\eta,\eta) = \frac{(\xi-\eta)\cdot \eta}{(e^{|\xi|}+e^{-|\xi|})(e^{|\eta|}+e^{-|\eta|})} \int_{-1}^{0}[e^{(s+1)|\xi|} + e^{-(s+1)|\xi|}][e^{(s+1)|\eta|} +e^{-(s+1)|\eta|}] d s  
\]
\be\label{eqnn3}
= \frac{(\xi-\eta)\cdot \eta}{(e^{|\xi|}+e^{-|\xi|})(e^{|\eta|}+e^{-|\eta|})} \Big[ \frac{e^{|\xi|+|\eta|}-e^{-|\xi|-|\eta|}}{|\xi|+|\eta|} + \frac{e^{|\xi|-|\eta|} - e^{|\eta|-|\xi|}}{|\xi|-|\eta|} \Big].
\ee
\[
q_{2,1}(\xi-\eta, \eta) = \frac{-2|\xi||\eta|}{(e^{|\xi|}+e^{-|\xi|})(e^{|\eta|}+e^{-|\eta|})}\int_{-1}^0 [e^{(s+1)|\xi|} - e^{-(s+1)|\xi|}][e^{(s+1)|\eta|} -e^{-(s+1)|\eta|}]  d s
\]
\begin{equation}\label{eqn80}
= \frac{-2|\xi||\eta|}{(e^{|\xi|}+e^{-|\xi|})(e^{|\eta|}+e^{-|\eta|})}\Big[  \frac{e^{|\xi|+|\eta|}-e^{-|\xi|-|\eta|}}{|\xi|+|\eta|} -  \frac{e^{|\xi|-|\eta|} - e^{|\eta|-|\xi|}}{|\xi|-|\eta|} 
\Big],
\end{equation}
\[
q_{2,2}(\xi-\eta,\eta)=  \frac{|\xi|(\xi-\eta)\cdot \eta}{(e^{|\xi|}+e^{-|\xi|})(e^{|\eta|}+e^{-|\eta|})} \int_{-1}^0 (s+1)[e^{(s+1)|\xi|}-e^{-(s+1)|\xi|}][e^{(s+1)|\eta|} + e^{-(s+1)|\eta|}]  d s
\]
\[
=  \frac{|\xi|(\xi-\eta)\cdot \eta}{(e^{|\xi|}+e^{-|\xi|})(e^{|\eta|}+e^{-|\eta|})}\Big[\frac{(|\xi|+|\eta|-1)e^{|\xi|+|\eta|}-(-|\xi|-|\eta|-1)e^{-|\xi|-|\eta|}}{(|\xi|+|\eta|)^2}
\]
\begin{equation}\label{eqn81}
- \frac{(|\eta|-|\xi|-1)e^{|\eta|-|\xi|}-(|\xi|-|\eta|-1)e^{|\xi|-|\eta|}}{(|\xi|-|\eta|)^2}\Big].
\end{equation}
 In above computations, we used the  following simple fact,
\[
\int_{-1}^{0} (s+1) e^{(s+1) a} d s = \frac{1+(a-1)e^{a}}{a^2}.
\]
From  (\ref{eqqn1}), (\ref{eqn40}), (\ref{eqn2008}), (\ref{eqnn2}) and (\ref{eqnn6}), we have
\[
\Lambda_2[G(\,h)\psi]:=\widetilde{Q}( h, \psi)= Q_1(h,\psi) + Q_2(h,\psi) + h \d \tanh\d \psi.
\]
Therefore, the symbol $\tilde{q}(\xi-\eta, \eta)$ of $\widetilde{Q}( h, \psi)$ is given as follows, 
\[
\tilde{q}(\xi-\eta, \eta)=\sum_{i,j=1,2} q_{i,j}(\xi-\eta, \eta) +\frac{e^{|\eta|} - e^{-|\eta|}}{e^{|\eta|}+e^{-|\eta|}}|\eta|. 
\]
Although above formulas look complicated, actually there are cancellations inside. Note that,
\[
 q_{1,2}(\xi-\eta, \eta)  +  q_{2,1 }(\xi-\eta, \eta) +\frac{e^{|\eta|} - e^{-|\eta|}}{e^{|\eta|}+e^{-|\eta|}}|\eta| 
\]
\[
= \frac{\xi\cdot \eta}{(e^{|\xi|}+e^{-|\xi|})(e^{|\eta|}+e^{-|\eta|})} \Big[ \frac{e^{|\xi|+|\eta|}-e^{-|\xi|-|\eta|}}{|\xi|+|\eta|}+ \frac{e^{|\xi|-|\eta|} - e^{|\eta|-|\xi|}}{|\xi|-|\eta|} \Big]  
\]
\begin{equation}\label{eqn82}
- \frac{|\xi||\eta|}{(e^{|\xi|}+e^{-|\xi|})(e^{|\eta|}+e^{-|\eta|})}\Big[ \frac{[e^{|\xi|+|\eta|}-e^{-|\xi|-|\eta|}]}{|\xi|+|\eta|}- \frac{e^{|\xi|-|\eta|}-e^{|\eta|-|\xi|}}{|\xi|-|\eta|}\Big],
\ee
\[
 q_{1,1}(\xi-\eta, \eta)  +  q_{2,2 }(\xi-\eta, \eta) 
\] 
\[
=\frac{ (-|\xi||\eta|+\xi\cdot\eta)}{(e^{|\xi|}+e^{-|\xi|})(e^{|\eta|}+e^{-|\eta|})} \frac{(|\xi|+|\eta|-1)e^{|\xi|+|\eta|}-(-|\xi|-|\eta|-1)e^{-|\xi|-|\eta|}}{(|\xi|+|\eta|) }
\]
\be\label{eqnn989}
- \frac{ ( \xi||\eta|+\xi\cdot\eta)}{(e^{|\xi|}+e^{-|\xi|})(e^{|\eta|}+e^{-|\eta|})} \frac{(|\eta|-|\xi|-1)e^{|\eta|-|\xi|}-(|\xi|-|\eta|-1)e^{|\xi|-|\eta|}}{(|\xi|-|\eta|) }\Big].
\ee
From (\ref{eqn82}) and (\ref{eqnn989}),  now  it is easy to verify the following equality,
 \be\label{eqnq929}
\tilde{q}(\xi-\eta, \eta)= \xi\cdot \eta - |\xi||\eta| \tanh|\xi|\tanh|\eta|.
 \ee
 Hence our desired equality (\ref{eqq345}) holds.

\qed

\begin{lemma}\label{symbol}
For $k_1,k_2,k\in \mathbb{Z}$,  the following estimate holds for the symbol of the quadratic terms of the Dirichlet-Neumann operator,
\begin{equation}\label{eqn1202}
\| \widetilde{q}(\xi-\eta, \eta)\|_{\mathcal{S}^\infty_{k,k_1,k_2}} \lesssim 2^{k+k_2}.
\end{equation}
\end{lemma}
\begin{proof}
From  (\ref{eqnq929}) and the estimate in    Lemma \ref{Snorm}, it is straightforward to derive above estimate. 
\end{proof}

\subsection{A fixed point type formulation for $\Lambda_{\geq 3}[\nabla_{x,z}\varphi]$}\label{fixedpointcubic} 
Same as in the previous subsection, the content of this subsection is not related to the proof of the main theorem but related to the future study of the long time behavior of the water waves system in different settings.

Although, intuitively speaking, the quadratic terms are the leading terms for the dispersive equation (\ref{dispersive}) in $2D$ case, one also has to control the cubic and higher order remainder terms to see that their effects are indeed small over time.  In this subsection, our goal is to formulate $\Lambda_{\geq 3}[\nabla_{x,z}\varphi]$ into a fixed point type formulation, which provides a good way to estimate the cubic and higher order remainder terms.

 Recall the fixed point type formulation of $\nabla_{x,z}\varphi$ in (\ref{fixedpoint}), we truncate it at the cubic and higher level and have the following, 
\[
\Lambda_{\geq 3}[\nabla_{x,z}\varphi]=[\mathbf{0}, \Lambda_{\geq 3}[g_1(z)]] +
\]
\[
+\int_{-1}^{0} [K_1(z,s)-K_2(z,s)-K_3(z,s)](\Lambda_{\geq 3}[g_2(s)]+\nabla \cdot \Lambda_{\geq 3}[g_3(s)])  ds \]
\begin{equation}\label{eqn68}
+\int_{-1}^{0} K_3(z,s)\d\textup{sign($z-s$)}\Lambda_{\geq 3}[g_1(s)]  -\d [K_1(z,s) +K_2(z,s)]\Lambda_{\geq 3}[g_1(s)]\, d  s.
\end{equation}
Recall (\ref{eqn782})  and (\ref{eqn783}).  Similar to the decomposition we did in (\ref{eqn2001}), (\ref{eqn2002}), (\ref{eqn2003}) and (\ref{eqn2004}), we can separate  $\Lambda_{\geq 3}[g_i(z)]$, $i\in\{1,2,3\}$, into two parts: (i) one of them contains $\Lambda_{\geq 3}[\nabla_{x,z}\varphi]$, in which lies the fixed point structure; (ii) the other part does not depend on $\Lambda_{\geq 3}[\nabla_{x,z}\varphi]$, which can be estimated directly. 
 
 More precisely, we  decompose $\Lambda_{\geq 3}[g_i(z)], i\in\{1,2,3\},$ as follows,
\begin{equation}\label{eqn50}
\Lambda_{\geq 3}[g_1(z)] = \tilde{\,h}_1 \Lambda_{\geq 3}[\p_z \varphi] -  (z+1)^2  \tilde{\,h}_2 \Lambda_{
\geq 3}[\p_z\varphi] + (z+1) \tilde{\,h}_3\cdot \Lambda_{\geq 3}[\nabla \varphi]
\end{equation}
\begin{equation}\label{eqn2000}
+\sum_{i=1,2} \Lambda_{\geq 3-i}[\tilde{\,h}_1] \Lambda_{  i}[\p_z \varphi]  -  (z+1)^2  \Lambda_{\geq 3-i}[\tilde{\,h}_2 ] \Lambda_{i}[\p_z\varphi]  + (z+1) \Lambda_{\geq 3-i}[\tilde{\,h}_3]\cdot \Lambda_{i}[\nabla \varphi] ,
\end{equation}
\[
\Lambda_{\geq 3}[g_2(z)] = (z+1)  \tilde{\,h}_2  \Lambda_{\geq 3}[\p_z \varphi] - \tilde{\,h}_3\cdot \Lambda_{\geq 3}[\nabla \varphi]
\]
 \begin{equation}\label{eqn790}
+\sum_{i=1,2}(z+1)\Lambda_{\geq 3-i}[ \tilde{\,h}_2] \Lambda_{i}[\p_z \varphi] -   \Lambda_{\geq 3-i}[\tilde{\,h}_3]\cdot \Lambda_{i }[\nabla \varphi] ,
\end{equation}
\begin{equation}\label{eqn53}
\Lambda_{\geq 3}[g_3(z)] =  (z+1)  \tilde{\,h}_3 \Lambda_{\geq 3}[\p_z \varphi]  + \sum_{i=1,2}( z+1)  \Lambda_{\geq 3-i}[\tilde{\,h}_3] \Lambda_{ i}[\p_z \varphi]  .
\end{equation}
From (\ref{tildeeta}), it is easy to verify that the following equalities hold,
\[
\Lambda_{\geq 2}[\tilde{h}_1]= h^2 -(2h+h^2) \tilde{h}_1, \quad \Lambda_{\geq 2}[\tilde{h}_2]= \tilde{h}_2,\quad \Lambda_{\geq 2}[\tilde{h}_3]= -h \tilde{h}_3.
\]

We can summarize above decomposition as the following Lemma, 
\begin{lemma}
We have
\[
\Lambda_{\geq 3}[\nabla_{x,z}\varphi(z)] = \sum_{i=1,2,3} C_z^i (\,h, \psi, \tilde{\,h}_i) + \,h \tilde{C}_z^i(\,h, \psi, \tilde{\,h}_i)+  T^i_z(\tilde{\,h}_i, \Lambda_{\geq 3}[\nabla_{x,z}\varphi]),\]
 where $C^i_z$ and $\tilde{C}_z^i$  are some trilinear operators and $ T^i_z$ is some bilinear operator. Assume that the corresponding symbols are   $c_z^i(\cdot, \cdot, \cdot)$, $\tilde{c}_z^i(\cdot, \cdot, \cdot)$, and $t_z^i(\cdot, \cdot)$ respectively,  then   the following estimate holds,
 \begin{equation}\label{eqn2304}
 \sup_{z\in[-1,0]}\sum_{i=1,2,3} \|c_{z}^i(\xi_1, \xi_2, \xi_3)\|_{\mathcal{S}^\infty_{k,k_1,k_2,k_3}} + \|\tilde{c}_z(\xi_1, \xi_2,\xi_3)\|_{\mathcal{S}^\infty_{k,k_1,k_2,k_3}}  \lesssim  2^{3\max\{k_1,k_2,k_3\}_{+}},
\end{equation}
\begin{equation}\label{eqn2306}
\sup_{z\in[-1,0]} \sum_{i=1,2,3} \|t^{i}_z(\xi_1,\xi_2)\|_{\mathcal{S}^\infty_{k,k_1,k_2}}  \lesssim 2^{3\max\{k_1,k_2\}_{+}}.
\end{equation}

\end{lemma}
\begin{proof}
The proof is   straightforward. From   Lemma \ref{Snorm}, our desired estimates (\ref{eqn2304}) and (\ref{eqn2306}) can be derived by checking symbol of  each term inside the equations (\ref{eqn68}), (\ref{eqn2000}), (\ref{eqn790}) and (\ref{eqn53}). Note that there are at most three derivatives in total.
\end{proof}

\section{Paralinearization and symmetrization of the system}\label{parasymm}

Since the gravity waves system (\ref{waterwave}) is quasilinear and lacks symmetric structures inside, we cannot use this system directly to do energy estimate because of the difficulty of losing one derivative.

To identity the hidden symmetries inside the gravity waves system (\ref{waterwave}) and get around this losing derivative issue, we use the method of  paralinearization and symmmetrization, which was   introduced and studied in a series of works by Alazard-M\'etivier \cite{alazard4} and Alazard-Burq-Zuily \cite{alazard1,alazard2,alazard3}. Interested readers may refer to those works for more details. Here, we  only briefly discuss this method here to help readers understand how this method works and get a sense of what they will read in this section.

For a fully nonlinear term, it is very hard to tell which part   actually loses derivatives and which part doesn't, which is clearly very important to get around the issue of losing derivatives. With the help of  the paralinearization process, we can identify the part that  actually loses derivatives, which is the real issue. In subsection \ref{paralinearDN}, we will do the paralinearization process for the nonlinearity of equation satisfied by the height ``$h$'', which is the Dirichlet-Neumann operator. In subsection \ref{paralinearvelocitypotential},   
we will do the paralinearization process for the nonlinearity of equation satisfied by the velocity potential ``$\psi$''.

Know which part loses derivative is certainly very helpful, but it doesn't imply that we can get around the issue of losing derivatives because the original system lacks good symmetric structures. With the help of the symmetrization process, in subsection \ref{symmetrizationsubsection}, we  identify good substitution variables such that the system of equations satisfied by the good substitution variables have requisite symmetries and moreover it has comparable size of energy as the original variables. Therefore, instead of doing energy estimate for the original variables, we do energy estimate for the good substitution variables.

\subsection{Paralinearization of the Dirichlet-Neumann operator}\label{paralinearDN}
In this subsection, our main goal is to identify which part of the Dirichlet-Neumann operator actually loses derivatives by using the paralinearization method. In the mean time, we also pay attentions to the low frequency part for the purpose of proving our new energy estimate (\ref{energyestimate}).

More precisely, the goal of this subsection is to prove the   Proposition as follows,
 
\begin{proposition}\label{prop1}
Let $k\geq 6$, $\alpha\in(0,1]$. Assume that $(\,h, \Lambda \psi)\in H^{k}$ and $\,h$ satisfies the smallness condition \textup{(\ref{smallness})}, then we have
\begin{equation}\label{paralinear}
G(\,h)\psi = T_{\lambda} (\psi - T_{B}\,h) - T_{V}\cdot \nabla \,h + F(\,h) \psi = \Lambda^2(\psi - T_{B} \,h) + T_{\lambda -|\xi|}(\psi-T_{B}\,h) - T_{V}\cdot \nabla \,h + \widetilde{F}(\,h)\psi,
\end{equation}
where
\begin{equation}\label{principalsymbol}
\lambda:= \sqrt{(1+|\nabla \,h|^2)|\xi|^2 - (\nabla \,h \cdot \xi)^2 },
\end{equation}
\begin{equation}\label{hvderivative}
B \overset{\text{abbr}}{=} B(\,h)\psi =\frac{G(\,h)\psi + \nabla \,h \cdot \nabla \psi}{1+|\nabla \,h|^2},\quad V\overset{\text{abbr}}{=} V(\,h)\psi= \nabla \psi - B \nabla \,h .
\end{equation}
The good remainder terms $F(\,h)\psi$ and $\widetilde{F}(h)\psi$ don't lose derivatives and they satisfy the following estimate,
\begin{equation}\label{goodremainder}
\|\Lambda_{\geq 2}[ F(\,h)\psi]\|_{H^{k}}+ \|\widetilde{ F}(\,h)\psi\|_{H^{k}}\lesssim_k \big[\|(\,h, \Lambda\psi)\|_{\widehat{W^{4,\alpha}}} + \| (\,h,\Lambda\psi)\|_{\widehat{W}^{4}}^2\big]
\big[ \|\,h\|_{H^k} + \|\nabla\psi\|_{H^{k-1}} \big].
\end{equation}

\end{proposition}

\begin{remark}
We remark that, unlike the infinite depth setting, the good remainder term $F(\,h)\psi$ in (\ref{paralinear}) actually contains a linear term, which is $[\tanh(|\nabla|)-1]|\nabla|\psi\in H^{\infty}$.
\end{remark}

For simplicity, we define the following equivalence relation. For two well defined nonlinearities   $A$ and $B$, which are nonlinear with respect to $\,h$ and $\psi$, we  denote
\[
A\thickapprox B, \quad \textup{if and only if $A-B$ is a good error term in the sense of (\ref{gooderror})},
\]
\begin{equation}\label{gooderror}
\|\textup{good error term}\|_{H^{k}}\lesssim_{k} \big[\|(\,h, \Lambda\psi)\|_{\widehat{W^{4,\alpha}}} + \| (\,h,\Lambda \psi)\|_{\widehat{W}^{4}}^2\big]
\big[ \|\,h\|_{H^k} + \|\nabla\psi\|_{H^{k-1}} \big],\,\, \alpha\in(0,1], k\geq 0.
\end{equation}

Recall (\ref{DN2}). Note that, essentially speaking, the only fully nonlinear term inside the Dirichlet-Neumann operator $G(h)\psi$ is $\p_w \Phi\big|_{w=0}$. So the task is reduced to identify which part of $\p_w \Phi$ is actually losing derivative. 

To this end, we will show that there exists a pseduo differential operator $A(x, \xi)$ such that $\p_w \Phi -T_A(\Phi -T_{\p_w \Phi} h)$ actually doesn't lose derivatives, where `` $\Phi- T_{\p_w \Phi}  h$'' is the so-called good unknown variable. This step is very nontrivial and technical. Unfortunately, to the best of my knowledge, there is no physical intuitive explanation available.   It relies  heavily on the study of good structures of the Laplace equation (\ref{changco2}). We do this step in details in the subsubsection as follows.

\subsubsection{Paralinearization of the Laplace equation \textup{(\ref{changco2})} }

 Recall (\ref{changco2}) and the fact that $g'(\cdot, w)=0$ when $w\in[-1/4,0]$, we have
\begin{equation}\label{eqn710}
[ \Delta_{x}+ a'\p_w^2 +  b'\cdot \nabla \p_w + c'\p_w ]\Phi = 0,
\end{equation}
\[
a'= 1 + |\nabla \,h|^2\ta 1+2 T_{\nabla \,h}\cdot \nabla \,h, \quad b'= -2\nabla \,h, \quad c'= -\Delta \,h.
\]
We remark that  $w$  is  also restricted inside $[-1/4,0]$ in the rest of this paper.

Before proceeding to do the paralinearization process for (\ref{eqn710}), we need some necessary estimates of $``\Phi"$. Essentially speaking, under certain smallness condition, the size of `'$\Phi$'' is comparable to $``\varphi"$. Note that  we already have necessary estimates of $\varphi$, see Lemma \ref{Sobolevestimate}.
 More precisely,
from definition of $\Phi$ (see subsection \ref{type2}) and estimates of $\varphi$ in  Lemma \ref{Sobolevestimate}, the following lemma holds,
\begin{lemma}\label{IntermsofPhi}
Under the smallness estimate \textup{(\ref{smallnessestimate})}, we have the following estimates for $k\geq 1,\gamma \leq 3,$
\begin{equation}\label{eqn3123}
\sup_{w\in[-1/4, 0]} \| \nabla_{x,w}\Phi\|_{H^{k}}\lesssim \|\nabla\psi\|_{H^k} + \| \,h\|_{H^{k+1}}\|\nabla \psi\|_{\widetilde{W^1}},\end{equation}
\begin{equation}\label{eqn3124}
\sup_{w\in[-1/4,0]}\| \nabla_{x}\Phi\|_{\widetilde{W^\gamma}} \lesssim \| \nabla \psi \|_{\widetilde{W^{\gamma}}} , \quad \sup_{w\in[-1/4,0]}\| \p_w \Phi\|_{\widetilde{W^{\gamma}}}\lesssim \| \nabla \psi\|_{\widehat{W}^{\gamma,\alpha}}+  \| \,h\|_{\widetilde{W^{\gamma+1}}} \| \nabla \psi \|_{\widetilde{W^\gamma}},
\end{equation}
\begin{equation}\label{eqn3125}
\sup_{w\in[-1/4,0]} \| \Lambda_{\geq 2}[\nabla_{x,w} \Phi] \|_{L^2}\lesssim\big[ \|(\,h,\Lambda \psi)\|_{\widehat{W^{2,\alpha}}}+\|(\,h, \Lambda \psi)\|_{\widehat{W^2}}^2\big]\big[ \|\,h\|_{H^1} + \|\nabla \psi\|_{L^2} \big].
\end{equation}
\end{lemma}
\begin{proof}
Postponed to the end of this subsection for the purpose of improving presentation.
 \end{proof}

After paralinearizing (\ref{eqn710}), we have
\begin{equation}\label{eqn712}
P\Phi + 2T_{\p_{w}^2 \Phi} T_{\nabla\,h}\cdot \nabla \,h -2 T_{\nabla \p_{w}\Phi} \cdot \nabla \,h - T_{\p_w \Phi} \Delta \,h \ta 0,
\end{equation}
where
\begin{equation}\label{operatorP}
P:=[\Delta + T_{a'}\p_{w}^2 + T_{b'\cdot } \nabla\p_w + T_{c'} \p_w].
\end{equation}
To see  why the equivalence relation (\ref{eqn712}) holds, we mention that we can always put $\nabla h$ in $L^\infty$ and   put $\p_w\Phi$ and $\p_w^2  \Phi$ in $L^2$. 

 Define $W:= \Phi - T_{\p_w \Phi}\,h$. Same as in \cite{alazard1}, we claim that $P W \ta 0$ when $w\in[-1/4,0]$.  After using (\ref{equation340}) and the composition Lemma \ref{composi}, the following equivalence relations hold, 
\begin{equation}\label{eqn720}
P W \ta 0  \Longleftrightarrow P [T_{\p_{w}\Phi} \,h] +  2T_{\p_{w}^2 \Phi} T_{\nabla\,h}\cdot \nabla \,h -2 T_{\nabla \p_{w}\Phi} \cdot \nabla \,h - T_{\p_w \Phi} \Delta \,h\ta  0,
\end{equation}
\begin{equation}\label{eqn721}
\Longleftrightarrow \big[T_{a'} T_{\p_{w}^3 \Phi} \,h + T_{b'}\cdot \nabla T_{\p_{w}^2\Phi} \,h + T_{c'}T_{\p_{w}^2\Phi} \,h + 2T_{\nabla\p_w \Phi} \cdot \nabla \,h  \big] + \big[ 2 T_{\p_{w}^2 \Phi} T_{\nabla\,h}\cdot \nabla \,h -2 T_{\nabla \p_{w}\Phi} \cdot \nabla \,h  \big]\ta 0
\end{equation}
\begin{equation}\label{eqn722}
\Longleftrightarrow  \big[ T_{b'}\cdot T_{\p_{w}^2 \Phi}\nabla \,h +  2T_{\nabla\p_w \Phi} \cdot \nabla \,h \big] + \big[ 2T_{\p_{w}^2 \Phi} T_{\nabla\,h}\cdot \nabla \,h -2 T_{\nabla \p_{w}\Phi} \cdot \nabla \,h \big]\ta 0
\end{equation}
\begin{equation}\label{eqn723}
\Longleftrightarrow T_{[b'\p_{w}^2 \Phi + 2\p_{w}^2\Phi \nabla \,h ]}\cdot \nabla \,h \ta 0 \end{equation}
\begin{equation}\label{eqn742}
\Longleftrightarrow 0 \ta 0, \quad \textup{as $b'=-2\nabla \,h$}.
\end{equation}
Obviously, (\ref{eqn742}) holds. Hence, we can reverse the directions of all arrows back to conclude $P W \ta 0$. 

\par Although tedious, it is  not  difficult to verify  all  ``$\ta$ " equivalence relations  hold in above equations. As a typical example, we  give a detail proof of (\ref{eqn721}) here. To prove (\ref{eqn721}), it would be  sufficient to estimate $T_{\Delta\p_{w}\Phi} h$. From estimate (\ref{eqn3124}) in Lemma \ref{IntermsofPhi}, we have,
\[
\sup_{w\in[-1/4,0]} \| T_{\Delta \p_w \Phi} \,h\|_{H^{k}}
\lesssim \sup_{w\in[-1/4,0]} \| \,h\|_{H^k} \|\p_w\Phi\|_{\widetilde{W^2}}\lesssim [ \|\Lambda\psi\|_{\widehat{W}^{4,\alpha}} +\|(\,h,\Lambda\psi)\|_{\widehat{W}^4}^2]\|\,h\|_{H^k}.
\]
Hence, the equivalence relation (\ref{eqn721}) holds. All other equivalence relations can be obtained very similarly.

The next step is to decompose the equation $P W \ta 0$ into a forward evolution equation  and  a backward evolution equation, which shows that $\p_w W-T_A W  $ actually doesn't lose derivatives. Note that  $\p_w W-T_A W\approx \p_w \Phi -T_A(\Phi -T_{\p_w \Phi} h) $. Hence, our desired result is obtained.

More precisely, we have the following Lemma,   
\begin{lemma}\label{decomp}
There exist two symbols $a=a(x,\xi)$ and $A(x,\xi)$ with 
\[
a= a^{(1)}+ a^{(0)}, \quad  A = A^{(1)} + A^{(0)}, 
\]
where
\[
a^{(1)}(x,\xi)= \frac{1}{1+|\nabla \,h|^2}(i \nabla \,h\cdot \xi - \sqrt{(1+|\nabla \,h|^2)|\xi|^2-(\nabla\,h \cdot\xi)^2}),
\]
\be\label{eqqq098}
A^{(1)}(x,\xi)=  \frac{1}{1+|\nabla \,h|^2}(i \nabla \,h\cdot \xi + \sqrt{(1+|\nabla \,h|^2)|\xi|^2-(\nabla\,h \cdot\xi)^2}),
\ee
\[
a^{(0)}(x,\xi)= \frac{1}{A^{(1)}- a^{(1)}}\Big(i \p_{\xi}a^{(1)}\cdot \p_x A^{(1)} - \frac{\Delta \,h \,\,a^{(1)}}{1+|\nabla \,h|^2}\Big),
\]
\[
A^{(0)}(x,\xi)=  \frac{1}{a^{(1)}- A^{(1)}}\Big(i \p_{\xi}a^{(1)}\cdot \p_x A^{(1)} - \frac{\Delta \,h \,\,A^{(1)}}{1+|\nabla \,h|^2}\Big),
\]
such that 
\begin{equation}\label{eqn1450}
P = T_{a'}(\p_w - T_{a})(\p_w -T_{A})  + R_0 + R_1\p_w,\quad a'(a+A)= i b'\cdot \xi + c',
\end{equation}
\begin{equation}\label{eqn1452}
a'[a^{(1)}A^{(1)} + \frac{1}{i} \p_{\xi}a^{(1)} \cdot\p_x A^{(1)} + a^{(1)} A^{(0)} + a^{(0)}A^{(1)}] = a'(a\sharp A)= -|\xi|^2, 
\end{equation}
where
\begin{equation}\label{equation210}
R_0 = T_{a'}T_{a}T_A - \Delta, \quad R_1 = - T_{a'} T_{a+A} + T_{b'}\cdot \nabla + T_{c'}.
\end{equation}
Moreover, the following estimate holds for good error operators $R_0$ and $R_1$,
\begin{equation}\label{errorestimate}
\| R_0 f \|_{H^k} + \| R_1 f\|_{H^{k+1}}\lesssim \| \nabla \,h\|_{\widetilde{W^3}} \| f\|_{H^k}. 
\end{equation}
\end{lemma}
\begin{proof}
Most parts of above Lemma are cited directly from \cite{alazard1}[Lemma 3.18]. Given the  a-priori decomposition (\ref{eqn1450}), from (\ref{operatorP}), we can calculate explicitly the formulas of  $R_0$ and $R_1$, which are given in  (\ref{equation210}).  Note that  $a'$ doesn't depend on $\xi$, from (\ref{eqn1450}), (\ref{eqn1452}) and (\ref{equation210}), we have the following identities, 
\[
R_1= -T_{a'}T_{a+A} +T_{a'(a+A)}= -T_{a'}T_{a+A}+ T_{(a'\sharp (a+A))},\]
\[
R_0= T_{a'}[T_{a}T_{A}-T_{a\sharp A}] + T_{a'}T_{a\sharp A} - T_{a'(a\sharp A)}= T_{a'}[T_{a}T_{A}-T_{a\sharp A}] + T_{(a'-1)}T_{a\sharp A} - T_{(a'-1)\sharp(a\sharp A)} .
\]
From explicit formulations of $a'$, $a$ and $A$, we can see that $a', a'-1\in \Gamma^{0}_2(\R^2), a, A, a+A\in \Gamma^{1}_2(\R^2)$, $a\sharp A \in \Gamma^{2}_2(\R^2)$.  The following estimates  on their symbolic bounds hold, 
\[
M^2_2(a\sharp A)+M^{0}_{2}(a')\lesssim 1,\quad  M^{1}_2(a) +M^1_2(A) + M^1_2(a+A)\lesssim \|\nabla \,h\|_{\widetilde{W^3}}, \quad M_2^0(a'-1) \lesssim \|\nabla\,h\|_{\widetilde{W^3}}^2.
\]
From estimate (\ref{eqn700}) in Lemma \ref{composi}, we have
\[
\| R_1 f\|_{H^{k+1}} \lesssim M^{0}_2(a') M^{1}_2(a+A) \| f\|_{H^k}\lesssim \| \nabla \,h\|_{\widetilde{W^3}}\|f\|_{H^k},
\]
\[
\| R_0 f\|_{H^k}\lesssim \big[M_2^0(a') M^1_2(a)M_2^1(A) + M_2^0(a'-1) M_2^2(a\sharp A)\big] \| f\|_{H^k} \lesssim \| \nabla\,h\|_{\widetilde{W^3}}^2 \| f\|_{H^k}.
\]
Hence finishing the proof of (\ref{errorestimate}).
 \end{proof}
In the following Lemma, we prove that $\p_w W - T_A W$ doesn't lose derivative.
\begin{lemma}\label{doesntlosederivative}
Let $A(x,\xi)$ be defined in Lemma \textup{\ref{decomp}}, we have  the following estimate for $k\geq 1,$
\be\label{eqqq675}
\| \Lambda_{\geq 2}\big[(\p_w W - T_A W)\big]\big|_{w=0}\|_{H^k}\lesssim_k  \big[\|(\,h, \Lambda \psi)\|_{\widehat{W^{4,\alpha}}} +  \|(\,h, \Lambda \psi)\|_{\widehat{W^{4}}}^2 \big]\big[ \|\,h\|_{H^k} +\|\nabla \psi\|_{H^{k-1}} \big].
\ee
\end{lemma}
\begin{proof}

Recall that $P W \ta 0$ and the decomposition of operator $P$ in  (\ref{eqn1450}), we have
\[
T_{a'} (\p_w - T_a )(\p_w - T_{A}) W \ta - R_0 W - R_1 \p_w W, 
\]
which further gives us
\[
(\p_w - T_a )(\p_w - T_{A}) W  \ta \tilde{g}, \]
where
\[
\tilde{g}=T_{a'^{-1}}[- R_0 W - R_1 \p_w W] + [I- T_{a'^{-1}}T_{a'}](\p_w - T_a )(\p_w - T_{A}) W,
\]
\[
= T_{a'^{-1}}[- R_0 W - R_1 \p_w W] + [I-T_{1}+ T_{(a'-1)(a'^{-1}-1)}-T_{(a'-1)}T_{(a'^{-1}-1)}](\p_w - T_a )(\p_w - T_{A}) W.
\]
From estimate (\ref{errorestimate}) in Lemma \ref{decomp}, and the fact that $T_{(a'-1)(a'^{-1}-1)}-T_{(a'-1)}T_{(a'^{-1}-1)}$ is of order $-2$, we have
\[
\sup_{w\in[-1/4,0]} \| \Lambda_{\geq 2}[\tilde{g}(w)]\|_{H^k} \lesssim \|\nabla\,h\|_{\widetilde{W^3}} \big[ \|P_{\geq 1/2}[W]\|_{H^k}  + \| \p_{w} W\|_{H^{k-1}} \big]
\]
\[
\lesssim \big[\|(\,h, \Lambda \psi)\|_{\widehat{W^{4,\alpha}}} +  \|(\,h, \Lambda \psi)\|_{\widehat{W^{4}}}^2 \big]\big[ \|\,h\|_{H^k} +\|\nabla \psi\|_{H^{k-1}} \big].
\]
 Note that $[\p_w^2 + \Delta]\Lambda_1[W(w)]=[\p_w^2 + \Delta]\Lambda_1[\Phi(w)]=0 $ when $w\in[-1/4,0]$, see (\ref{eqn710}). It's easy to see the following equivalence relation holds,
\begin{equation}\label{eqn750}
 (\p_w - T_{a})\Lambda_{\geq 2}[(\p_w -T_{A}) W] + \Lambda_{\geq 2}[(\p_w - T_{a})\Lambda_{1}[(\p_w -T_{A}) W] ] \ta \Lambda_{\geq 2}[\tilde{g}].
\end{equation}
Note that
\[
\Lambda_{1}[\p_w \Phi]= \Lambda_{1}[\p_z \varphi(w/(1+\,h))]= \frac{e^{(w+1)\d}- e^{-(w+1)\d}}{e^{-\d}+e^{\d}} \d \psi,
\]
\[
\Lambda_{1}[\Phi]= \Lambda_{1}[\varphi(w/(1+\,h))]= \frac{e^{-(w+1)\d}+ e^{(w+1)\d}}{e^{-\d} + e^{\d}}\psi.
\]
It's easy to verify that,
\[
\Lambda_{1}[(\p_w -T_{A}) W] ]= \Lambda_{1}[\p_w \Phi - T_{|\xi|}\Phi] \in H^{\infty}, 
\]
\[
\| \Lambda_{\geq 2}[(\p_w - T_{a})\Lambda_{1}[(\p_w -T_{A}) W] ] \|_{H^k}\lesssim \| \nabla\,h\|_{\widetilde{W^3}}\| \nabla \psi\|_{H^{k-1}}.
\]
Therefore, from (\ref{eqn750}), we have
\[
 (\p_w - T_{a})\Lambda_{\geq 2}[(\p_w -T_{A}) W] \ta \Lambda_{\geq 2}[\tilde{g}].
 \]
We reformulate above equation as follows,
\[
(\p_w + T_{-a})\Lambda_{\geq 2}[(\p_w -T_{A}) W] = \Lambda_{\geq 2}[\tilde{g}] + \hat{g},
\]
where 
\[
\hat{g}= \textup{error term from ``$\ta$ equivalence" relation}.
\]
Recall the precise formula of ``$a$'' in Lemma \ref{decomp}, we know $-a$ satisfies the assumption of Lemma \ref{elliptic}. We can first choose a series of constants $\{\tau_i\}_{i=1}^k$ such that $\tau_{i+1}= 4\tau_{i}$ and $\tau_{k}\geq -1/5$ and then keep iterating the estimate (\ref{coreest1}). As a result, the following estimate holds, 
\[
 \| \Lambda_{\geq 2}[(\p_w -T_{A}) W\big|_{w=0}]\|_{H^{k}} \lesssim_{k} \sup_{w\in[-1/5,0]}\big[\| \Lambda_{\geq 2}[(\p_w -T_{A}) W(w,\cdot)]\|_{L^2}  + \|\tilde{g}(w)\|_{H^{k-1+\epsilon}}
\]
\begin{equation}\label{eqn730}
 + \| \hat{g}(w)\|_{H^{k-1+\epsilon}}\big]\lesssim_k  \big[\|(\,h, \Lambda \psi)\|_{\widehat{W^{4,\alpha}}} +  \|(\,h, \Lambda \psi)\|_{\widehat{W^{4}}}^2 \big]\big[ \|\,h\|_{H^k} +\|\nabla \psi\|_{H^{k-1}} \big].
\end{equation}
\end{proof}

\begin{lemma}\label{elliptic}
Let $a\in \Gamma^{1}_{2}(\R^2)$ and it satisfies the following assumption,
$
Re[a(x,\xi)]\geq c |\xi|,$
for some positive constant $c$. If $u$ solves the following equation 
\[
(\p_w + T_a) u(w,\cdot) = g(w,\cdot),
\]
then we have the following estimate holds for  any fixed and sufficiently  small constant  $\tau$, and arbitrarily small constant $\epsilon >0.$
\begin{equation}\label{coreest1}
\sup_{w\in[\tau,0]}\| u(w)\|_{H^{k}} \lesssim M_2^1(a) \frac{1+|\tau|}{|\tau|}\big[ \sup_{z\in[4\tau,0]} \| u(w)\|_{H^{k-2(1-\epsilon)}} + \sup_{z\in[4\tau, 0]} \| g(z)\|_{H^{\mu-(1-\epsilon)}}\big].
\end{equation}
\end{lemma}
\begin{proof}
A detailed proof can be found in \cite{alazard}. Above result is the combination of \cite{alazard}[Lemma 2.1.9] and the proof part of \cite{alazard}[Lemma 2.1.10].

\end{proof}

\subsubsection{Paralinearization of Dirichlet-Neumann operator} In this subsubsection, we use the result we obtained in last subsubsection, which is the fact that $\p_w W-T_A W$ doesn't lose derivatives, to identify which part of the Dirichlet-Neumann operator loses derivatives.

Recall  (\ref{DN2}). For readers' convenience, we  rewrite it as follows,
\[
G(\,h)\psi =\big( (1+|\nabla\,h|^2)\p_w \Phi - \nabla \,h \cdot \nabla \Phi\big) \big|_{w=0}.
\]
Define
\[
   \uline{V}:= \nabla \Phi- \p_w \Phi \nabla \,h,\quad  \uline{V}\big|_{w=0}=V.
\]

Now we   let $w$ be inside the range $[-1/4,0]$ instead of restricting it at the boundary. By using (\ref{equation340}) and   Lemma \ref{composi}, we have the paralinearization result as follows,
\[
(1+|\nabla\,h|^2)\p_w \Phi - \nabla \,h \cdot \nabla \Phi \ta T_{1+|\nabla \,h|^2}\p_w \Phi + 2T_{\p_w \Phi} T_{\nabla \,h}\cdot \nabla \,h - T_{\nabla \,h} \cdot \nabla \Phi - T_{\nabla\Phi}\cdot \nabla \,h
\]
\[
\ta T_{1+|\nabla \,h|^2}\p_w \Phi + T_{2\nabla\,h \p_w \Phi-\nabla \Phi}\cdot \nabla \,h - T_{\nabla \,h}\cdot \nabla \Phi
\]
\[
= T_{1+|\nabla \,h|^2}\p_w (W+T_{\p_w \Phi}\,h)+ T_{2\nabla\,h \p_w \Phi-\nabla \Phi}\cdot \nabla \,h - T_{\nabla \,h}\cdot \nabla (W+T_{\p_w \Phi}\,h)
\]
\[
\ta T_{1+|\nabla\,h|^2}\p_w W + T_{2\nabla \,h \p_w \Phi-\nabla\Phi}\cdot \nabla \,h - T_{\nabla \,h}\cdot \nabla W - T_{\nabla \,h \p_w \Phi}\cdot\nabla \,h 
= T_{1+|\nabla \,h|^2} \p_w W - T_{\nabla \,h}\cdot \nabla W - T_{\uline{V}}\cdot \nabla \,h,
\]
\[
= T_{1+|\nabla \,h|^2}[\p_w W - T_A W]+ [T_{1+|\nabla \,h|^2} T_{A} - T_{\nabla \,h}\cdot \nabla] W - T_{\uline{V}}\cdot \nabla \,h,
\]
\begin{equation}\label{eqn1520}
= T_{\lambda} W - T_{\uline{V}}\cdot \nabla \,h + T_{1+|\nabla \,h|^2}[\p_w W - T_A W] + R_2 W, \quad R_2:= [T_{1+|\nabla \,h|^2} T_A - T_{(1+|\nabla \,h|^2)A^{(1)}}],
\end{equation}
where $\lambda$ is given in (\ref{principalsymbol}). In (\ref{eqn1520}), we used  the following identity,
\[
\lambda = (1+|\nabla \,h|^2) A^{(1)} - i \xi\cdot \nabla \,h,
\]
where $A^{(1)} $ is given in (\ref{eqqq098}). Note that 
\[
R_2 = T_{(1+|\nabla\,h|^2)} T_{A^{(0)}} + \big[ T_{(1+|\nabla\,h|^2)} T_{A^{(1)}} - T_{(1+|\nabla \,h|^2)A^{(1)}}\big] = T_{a'} T_{A^{0}} + \big[ T_{(a'-1)} T_{A^{(1)}}- T_{(a'-1)\sharp A^{(1)}} \big].
\]
Now,  it is  easy to see that $R_2$ is an operator of order $0$ with an upper  bound given by $\| \nabla \,h\|_{\widetilde{W^3}}$. Hence, the following estimate holds, 
\[
\| R_2 W\big|_{w=0}\|_{H^{k}}= \| R_2 [P_{\geq 1/2}[W]\big|_{w=0}]\|_{H^{k}}\lesssim \| \nabla \,h\|_{\widetilde{W^{3}}}\| P_{\geq 1/2}[\psi - T_{B(\,h)\psi} \,h]\|_{H^{k}}
\]
\begin{equation}\label{eqqq09876}
\lesssim \big[\|(\,h, \Lambda \psi)\|_{\widehat{W^{4,\alpha}}} +  \|(\,h, \Lambda \psi)\|_{\widehat{W^{4}}}^2 \big]\big[ \|\,h\|_{H^k} +\|\nabla \psi\|_{H^{k-1}} \big].
\end{equation}
Combining (\ref{eqn1520}), (\ref{eqqq09876}), and   the estimate (\ref{eqqq675}) in Lemma \ref{doesntlosederivative}, it's easy to see Proposition \ref{prop1} holds.

\vo 

Now, we give the postpone proof  of Lemma \textup{\ref{IntermsofPhi}} as follows. 

\vo

\noindent\textit{Proof of Lemma \textup{\ref{IntermsofPhi}}.\quad}
For fixed $w\in[-1/4,0]$, it's easy to see $\Phi(w)= \varphi(w/(1+\,h(x)))$ and the  following identity holds, 
\begin{equation}\label{eqn3110}
\nabla_{x,w}\Phi= \nabla_{x,z}\varphi(w/(1+\,h)) + [-w \nabla\,h \p_z\varphi(w/(1+\,h))/(1+\,h)^2, -\,h \p_z \varphi(w/(1+\,h))/(1+\,h)].
\end{equation}
Therefore, we know that the leading term of $\nabla_{x,w}\Phi(w)$ is $\nabla_{x,z}\varphi(w/(1+\,h))$.  Under the smallness estimate (\ref{smallnessestimate}), to estimate $\nabla_{x,w}\Phi$, it's sufficient to estimate $\nabla_{x,z}\varphi(w/(1+\,h))$.  

Recall the fixed point type formulation of $\nabla_{x,z}\varphi$ in (\ref{fixedpoint}), we study the linear term on the right hand side of (\ref{fixedpoint}) first.
 Denotes 
\[
p_{\pm}(w,x,\xi):=  \pm \frac{e^{-(w+1)|\xi|}(e^{\,h(x)w|\xi|/(1+\,h(x))}-1)}{e^{-|\xi|} + e^{|\xi|}}+\frac{ e^{(w+1)|\xi|}(e^{-\,h(x)w|\xi|/(1+\,h(x))}-1) }{e^{-|\xi|} + e^{|\xi|}},
\]
\[
= \sum_{n\geq 1} \frac{1}{n!} \Big[ \pm \frac{e^{-(w+1)|\xi|}}{e^{|\xi|}+e^{-|\xi|}} (w|\xi|)^{n}\Big(\frac{\,h(x)}{1+\,h(x)}\Big)^{n}  + \frac{e^{|\xi|}}{e^{|\xi|}+e^{-|\xi|}} e^{w|\xi|}(w|\xi|)^{n}\Big( \frac{\,h(x)}{1+\,h(x)} \Big)^n\Big]
\]
\[
= \sum_{n\geq 1} \frac{1}{n!} \big[ \pm f^1_{n}(w,\xi) g_n(x) + f^2_n(w,\xi) g_n(x) \big],
\]
where
\[
f_n^1(w,\xi):=\frac{e^{-(w+1)|\xi|}}{e^{|\xi|}+e^{-|\xi|}} (w|\xi|)^{n},\quad f_n^2(w,\xi):=  \frac{e^{|\xi|}}{e^{|\xi|}+e^{-|\xi|}} e^{w|\xi|}(w|\xi|)^{n}, \quad g_n(x) = \Big( \frac{\,h(x)}{1+\,h(x)}\Big)^n.
\]
It's easy to verify the following identities hold,
\[
\Lambda_{1}[\nabla_{x}\varphi](x,w/(1+\,h(x)))= \Lambda_{1}[\nabla_{x}\varphi](x,w)+ P_{1} \psi(x,w),
\]
\[
\Lambda_{1}[\p_z\varphi](x,w/(1+\,h(x))) = \Lambda_{1}[\p_z\varphi](x,w) 
+ P_2 \psi (x,w),
\]
where
\[
 P_{1} \psi(x,w):=\frac{1}{4\pi^2}\int_{\R^2} e^{ix\cdot \xi }\widehat{\psi}(\xi) i\xi p_{+}(w,x,\xi) d \xi, \quad  P_{2} \psi(x,w)=\frac{1}{4\pi^2}\int_{\R^2} e^{ix\cdot \xi }\widehat{\psi}(\xi) |\xi| p_{-}(w,x,\xi) d \xi.
\]

We will show that, under the smallness estimate (\ref{smallnessestimate}),  the size of $\Lambda_{1}[\nabla_{x,z}\varphi](x,w/(1+\,h(x)))$ is almost same as the size of  $\Lambda_{1}[\nabla_{x,z}\varphi](x,w)$. For $ k \in \mathbb{Z}$, we define 
\begin{equation}\label{eqn3103}
p_{\pm, k}(w,x,\xi)=\frac{1}{4\pi^2} \int_{\R^2 } e^{ix\cdot \sigma}\mathcal{F}_{x}(p)(w,\sigma, \xi)\psi_{k}(\sigma) d\sigma = \frac{1}{4\pi^2}\sum_{n\geq 1} \frac{1}{n!}\big[ \pm f^1_{n}(w,\xi)  + f^2_n(w,\xi) \big]P_{k}[g_{n}](x),
\end{equation}
\[
 P_{1,k} \psi(x,w):=\int_{\R^2} e^{ix\cdot \xi }\widehat{\psi}(\xi) i\xi p_{+,k}(w,x,\xi) d \xi, \quad   P_{2,k} \psi(x,w)=\int_{\R^2} e^{ix\cdot \xi }\widehat{\psi}(\xi) |\xi| p_{-,k}(w,x,\xi) d \xi. 
\]
Since $P_{2}\psi$ can be treated in the same way as $P_1 \psi$,  we only estimate $P_1\psi$ in details here. We have the following decomposition,
\[
P_{1}\psi = \sum_{k_1, k_2\in \mathbb{Z}} P_{1,k_1} \psi_{k_2} = \mbox{I} + \mbox{II}, \quad \mbox{I}= \sum_{k_2 \leq k_1 } P_{1,k_2} \psi_{k_1}, \quad \mbox{II}= \sum_{k_1 \leq k_2} P_{1,k_2} \psi_{k_1}.
\]
From $L^2-L^\infty$ type bilinear estimate (\ref{bilinearesetimate}) in Lemma \ref{multilinearestimate}, it's easy to see the following estimates hold,
\[
\| \mbox{I} \|_{H^{k}} \lesssim \Big(\sum_{ k_1} 2^{2k_1 + 2k k_{1,+}} \| P_{k_1}\psi\|_{L^2}^2   \Big(\sum_{n\geq 1} \frac{1}{n!} \| P_{\leq k_1} g_{n}\|_{L^\infty}\Big)^2\Big)^{1/2} \lesssim \|\nabla \psi\|_{H^{k}} \| \,h\|_{\widetilde{W^1}},
\]
\[
\|\mbox{II}\|_{H^s} \lesssim  \sum_{n\geq 1} \frac{1}{n!} \sum_{k_2\leq k_1} 2^{k_2 + k k_{1,+}}  \| P_{k_1}  g_{n}\|_{L^2} \| P_{k_2}\psi\|_{L^\infty} \lesssim \| \,h\|_{H^{s+1}} \| \nabla \psi \|_{\widetilde{W^1}},
\]
\[
\| \mbox{I} \|_{\widetilde{W^\gamma}} \lesssim \|\,h\|_{\widetilde{W^\gamma}} \|\nabla\psi\|_{\widetilde{W^\gamma}},\quad \| \mbox{II}\|_{\widetilde{W^\gamma}} \lesssim \|\,h\|_{\widetilde{W^\gamma}} \|\nabla\psi\|_{\widetilde{W^\gamma}}.
\]
Therefore
\[
\|\Lambda_{1}[\nabla_{x,z}\varphi](x,w/(1+\,h(x)))\|_{H^k} \lesssim \| \nabla\psi\|_{H^k} + \|\,h\|_{H^{k+1}} \| \nabla \psi \|_{\widetilde{W}^{1}},
\]
\[
\|\Lambda_{1}[\nabla_{x}\varphi](x,w/(1+\,h(x)))\|_{\widetilde{W^\gamma}} \lesssim \|\nabla\psi \|_{\widetilde{W^\gamma}} + \|\,h\|_{\widetilde{W^\gamma}} \|\nabla\psi\|_{\widetilde{W^\gamma}}\lesssim \|\nabla \psi\|_{\widetilde{W^\gamma}},
\]
\[
\|\Lambda_{1}[\p_z \varphi](x,w/(1+\,h(x)))\|_{\widetilde{W^\gamma}} \lesssim \|\Lambda^2\psi \|_{\widetilde{W^{\gamma}}} + \|\,h\|_{\widetilde{W^\gamma}} \|\nabla\psi \|_{\widetilde{W^\gamma}}.
\]
From above estimates and  (\ref{eqn3110}), we conclude the following estimates,
\begin{equation}\label{eqn3127}
\|\Lambda_{1}[\nabla_{x}\Phi]\|_{\widetilde{W^\gamma}}\lesssim \|\nabla \psi\|_{\widetilde{W^\gamma}},\quad 
\|\Lambda_{1}[\p_w \Phi]\|_{\widetilde{W^\gamma}} \lesssim \|\Lambda^2\psi \|_{\widetilde{W^{\gamma}}} + \|\,h\|_{\widetilde{W^\gamma}} \|\nabla\psi \|_{\widetilde{W^\gamma}},
\end{equation}
\[
\| \Lambda_{1}[\nabla_{x,w}\Phi] \|_{H^k} \lesssim \|\Lambda_{1}[\nabla_{x,z}\varphi](x,w/(1+\,h)) \|_{H^k} + \|\Lambda_{1}[\nabla_{x,z}\varphi](x,w/(1+\,h))  \|_{\widetilde{W^0}}\|\,h\|_{H^{k+1}}
\]
\begin{equation}\label{eqn3120}
\lesssim \| \nabla \psi \|_{H^k} + \|\,h\|_{H^{k+1}}\|\nabla \psi\|_{\widetilde{W^1}}.
\end{equation}

Following a similar procedure, we can handle the integral part in (\ref{fixedpoint}) in the same way.  Similar to what we did in the proof of Lemma \ref{integraloperator},  we use the size of symbol directly when $|\xi|\leq 1$ and   estimate the associated kernel when $|\xi|\geq 1$. As a result, we have the following estimate,
\[
\|\textup{All terms in the R.H.S. of (\ref{fixedpoint}) except linear part}(x,w/(1+\,h(x)))\|_{H^{k}}
\]
\[
\lesssim \sum_{i=1,2,3}  \| g_i(z,\cdot)\|_{L^\infty_z H^k}  + \| \,h\|_{H^{k+1}} \| g_i(z,\cdot)\|_{L^\infty_z \widetilde{W^0}} \lesssim \|\nabla\psi \|_{\widetilde{W^0}}\|\,h\|_{H^{k+1}} + \| \nabla\psi\|_{H^k} \| \,h\|_{\widetilde{W^0}}
\]
\begin{equation}\label{eqn3121}
\lesssim \|\nabla\psi \|_{\widetilde{W^0}}\|\,h\|_{H^{k+1}} + \| \nabla\psi\|_{H^k} \| \,h\|_{\widetilde{W^0}}.
\end{equation}
\[
\|\textup{All terms in the R.H.S. of (\ref{fixedpoint}) except linear part}(x,w/(1+\,h(x)))\|_{\widetilde{W^\gamma}}
\]
\begin{equation}\label{equation346}
\lesssim \sum_{i=1,2,3}\| g_i(z,\cdot)\|_{L^\infty_z \widetilde{W^\gamma}} \lesssim \|\,h\|_{\widetilde{W^{\gamma+1}}} \|\nabla\psi\|_{\widetilde{W^\gamma}}.
\end{equation}
From (\ref{eqn3110}), (\ref{eqn3127}), (\ref{eqn3120}), (\ref{eqn3121}) and (\ref{equation346}), now it's easy to see estimates (\ref{eqn3123}) and (\ref{eqn3124}) hold.

Now, we proceed to prove (\ref{eqn3125}). From (\ref{eqn3110}) and the same procedure as above, we have the following estimate,
\[
\|\Lambda_{\geq 2}[\nabla_{x,w}\Phi]\|_{L^2} \lesssim \|\Lambda_{\geq 2}[\nabla_{x,z}\varphi](x,w/(1+\,h(x)))\|_{L^2}  + \|\,h\|_{H^1} \| \p_z\varphi(x,w/(1+\,h(x)))\|_{\widetilde{W^0}}
\]
\begin{equation}\label{eqn3200}
\lesssim \| \,h\|_{H^1} \big[\|\Lambda \psi \|_{\widehat{W^{2,\alpha}}} + \|\,h\|_{\widetilde{W^1}} \|\nabla \psi\|_{\widetilde{W^1}} \big] + \sum_{i=1,2,3}\|g_{i}\|_{L^\infty_z L^2}.
\end{equation}
Recall (\ref{eqn12}) and (\ref{eqn14}). Note that   $\nabla\varphi$ appears together with $\nabla h$ inside the quadratic terms of $g_i(z)$, $i\in\{1,2,3\}$.  When estimating the $L^{\infty}_{z} L^{2}$ norm of $g_{i}(z)$, $i\in\{1,2,3\}$, we  always put $\nabla\varphi$ in $L^2$ and put $\p_{z}\varphi$ in $L^\infty$. As a result,  the following estimate holds, i.e., our desired estimate (\ref{eqn3125}) holds,
\[
(\ref{eqn3200}) \lesssim \|\,h\|_{H^1}  \|\Lambda \psi \|_{\widehat{W^{2,\alpha}}} +  \big[\|(\,h, \Lambda\psi)\|_{\widehat{W^{2,\alpha}}}  + \|(\,h, \Lambda\psi)\|_{\widehat{W^2}}\big] (\|\,h\|_{H^1} +\|\nabla\psi\|_{L^2})\]
\[\lesssim \big[\|(\,h, \Lambda\psi)\|_{\widehat{W^{2,\alpha}}}  + \|(\,h, \Lambda\psi)\|_{\widehat{W^2}}\big] (\|\,h\|_{H^1} +\|\nabla\psi\|_{L^2} ).
\]

\qed

\subsection{Paralinearization of the equation satisfied by the velocity potential }\label{paralinearvelocitypotential}
In this subsection, our main goal is to do  paralinearization  process for the nonlinearity of equation satisfied by  $\psi$ in (\ref{waterwave}), which shows which part of nonlinearity actually loses derivatives. 

More precisely, the main result of this subsection is stated as the following Proposition, 
\begin{proposition}\label{para2}
We have the following paralinearization  for the nonlinearity of equation satisfied by $ \psi,$
\begin{equation}\label{eqn740}
\h |\nabla \psi|^2 - \h \frac{(\nabla\,h\cdot \nabla \psi + G(\,h)\psi)^2}{1+|\nabla \,h|^2}\ta T_{V}\cdot \nabla[ \psi -  T_{B(\,h)\psi} \,h] - T_{B}G(\,h)\psi.
\end{equation}
\end{proposition}

\begin{proof}
Recall that $V=\nabla \psi- \nabla \,h B$.  From (\ref{equation340}) and the composition Lemma \ref{composi}, we have
\[
\h |\nabla \psi|^2 - \h \frac{(\nabla\,h\cdot \nabla \psi + G(\,h)\psi)^2}{1+|\nabla \,h|^2} = \h |\nabla \psi|^2 - \h (1+|\nabla \,h|^2) B^2
\]
\[
= \h |V|^2 + V\cdot \nabla \,h B + \h |\nabla \,h|^2 B^2 -   \h (1+|\nabla \,h|^2) B^2= \h |V|^2 + V\cdot \nabla \,h B - \h B^2 
\]
\[
\ta T_{V}\cdot V + T_{B} (V\cdot \nabla \,h) + T_{V\cdot \nabla \,h} B - T_{B} B= T_{V}\cdot V + T_{V\cdot \nabla \,h} B - T_{B} G(\,h)\psi
\]
\[
=T_{V}\cdot \nabla \psi - T_{V}\cdot (\nabla \,h B) + T_{V\cdot \nabla \,h} B - T_{B} G(\,h)\psi
\]
\[
\ta T_{V}\cdot \nabla \psi  - T_{V}T_{B}\cdot \nabla \,h -[ T_{V}\cdot T_{\nabla \,h} - T_{V\cdot \nabla \,h} ]B - T_{B} G(\,h)\psi
\]
\[
\ta T_{V}\cdot [\nabla \psi - T_{B}\cdot \nabla \,h] - T_{B} G(\,h)\psi \ta T_{V}\cdot \nabla[ \psi -  T_{B(\,h)\psi} \,h] - T_{B} G(\,h)\psi.
\]
Hence, finishing  the proof of this proposition.
\end{proof}

\subsection{Symmetrization of the full system}\label{symmetrizationsubsection}
Based on the paralinearization results we obtained in previous subsections,  in this subsection, we will find out the good substitution variables by doing the symmetrization process  such that the resulted system has requisite symmetric structures inside. 

Define $\omega=\psi- T_{B(\,h)\psi} \,h$, which is the so-called good unknown variable.   After combining the good decomposition  (\ref{paralinear}) in Proposition \ref{prop1} and the good decomposition (\ref{eqn740}) in Proposition \ref{para2}, we reduce the system of equation satisfied by $h $ and $\psi$ to the system of equation satisfied by $h$ and $\omega$ as follows, 
\begin{equation}\label{presym}
\left\{\begin{array}{l}
\p_t \,h  =\Lambda^2 \omega+ T_{\lambda-|\xi|} \omega - T_{V}\cdot \nabla \,h + \widetilde{F}(\,h)\psi,\\
\\
\p_t \omega= -T_{a} \,h - T_{V}\cdot \nabla \omega  + f',
\end{array}\right.
\end{equation}
where
\[
a:=1+\p_t B+V\cdot \nabla B,
\]
which is the so-called Taylor coefficient  and $f'$ is a good error term  in the sense of estimate (\ref{gooderror}).

However,  the system (\ref{presym}) cannot be used to do energy estimate.   When using the system (\ref{presym}) to do the energy estimate, one might find that the following term loses one derivative and can't be simply treated,
\be\label{eqq2345}
\int_{\R^2} \p_{x}^{N}\,h \p_{x}^{N}[T_{\lambda-|\xi|} \omega] + \p_{x}^{N}\Lambda \omega \p_x^N \Lambda[-T_{a}\,h], \quad \textup{where $N$ is the prescribed top derivative level.}
\ee

To get around this difficulty,  we will symmetrize the system (\ref{presym}) by following the same procedures in \cite{alazard2}. Define 
\begin{equation}\label{transform}
U_1 = \,h + T_{\alpha} \,h, \quad U_2 = \Lambda [\omega + T_{\beta}\omega]\ta [T_{\sqrt{\lambda}-|\xi|^{1/2}}\omega] + \Lambda \omega, \quad U:= U_1 + i U_2,
\end{equation}
where
\be\label{taylorcoef}
\alpha= \sqrt{a}-1,\quad \quad \omega = \psi - T_{B(\,h)\psi}\,h, \quad
\beta := \sqrt{\lambda/|\xi|}-1 =\sqrt[4]{(1+|\nabla \,h|^2) - (\nabla\,h \cdot \xi/|\xi|)^2}-1.
\ee
Note that,
\[
\Lambda_{1}[ a ]= \Lambda_{1}[\p_t \Lambda^2 \psi] =- \Lambda^2 \,h, \quad \Lambda_1[\p_t  a] = -\Lambda^4 \psi,\quad \Lambda_1[\p_{t} \alpha] = -\h\Lambda^4 \psi,\quad \Lambda_1[\alpha] = -\h \Lambda^2 \,h.
\]

We   take the estimates of the Taylor coefficient in Lemma \ref{taylorcoefficient} as granted first. Then it is easy to see the following estimate and equivalence relations  hold,
\begin{equation}\label{eqn5900}
\| (U_1, U_2)- (\,h, \Lambda\psi)\|_{H^{k}}\lesssim \| (h, \Lambda \psi)\|_{\widetilde{W^{3}}} \|(\,h, \Lambda \psi)\|_{H^{k}}.
\end{equation}
\[
\p_t U_1\ta\Lambda^2 \omega+  [T_{\lambda-|\xi|} \omega] - T_{V}\cdot \nabla \,h  + T_{\p_t \alpha}\,h + T_{\alpha}[\p_t \,h]
\]
\[
\ta [\Lambda^2 -T_{|\xi|}] \omega+  T_{\lambda} \omega - T_{V}\cdot \nabla U_1 + T_{V}\cdot T_{\alpha} \nabla \,h  + T_{\alpha}\big[  T_{\lambda} \omega - T_{V}\cdot \nabla \,h \big]
\]
\[
\ta [\Lambda^2 -T_{|\xi|}] \Lambda^{-1} U_2 + T_{\lambda \sqrt{a}}\omega  - T_{V}\cdot \nabla U_1 
\ta \Lambda U_2 + [T_{\sqrt{ \lambda}\alpha} T_{\sqrt{\lambda}}\omega ]- T_{V}\cdot \nabla U_1
\]
\begin{equation}\label{equation200}
\ta \Lambda U_2 + [T_{\sqrt{\lambda}\alpha} U_2] - T_{V}\cdot \nabla U_1.
\end{equation}
\[
\p_t U_2 \ta \Lambda (1+T_{\beta}) [-T_{a} \,h - T_{V}\cdot \nabla \omega] + \Lambda T_{\p_t \beta} \omega 
\ta -\Lambda U_1 - [T_{\sqrt{\lambda}} T_{\alpha} U_1] - T_{V}\cdot \nabla T_{\sqrt{\lambda}} \omega + T_{\p_t \beta} \Lambda \omega\]
\begin{equation}\label{equation201}
\ta -\Lambda U_1 - [T_{\sqrt{\lambda}\alpha} U_1] - T_{V}\cdot \nabla U_2 + T_{\p_t \beta} U_2.
\end{equation}
Hence, the problematic terms in (\ref{eqq2345}) becomes the following terms (modulo good error terms),
\be\label{eqq2122}
\int_{\R^2} \p_{x}^{N} U_{1} \p_{x}^{N}[T_{\sqrt{\lambda}\alpha} U_2] - \p_{x}^{N}U_{2} \p_x^N   [T_{\sqrt{\lambda}\alpha} U_1], \quad \textup{where $N$ is the prescribed top derivative level.}
\ee
Therefore, we can move derivatives in (\ref{eqq2122}) around such that this cubic term doesn't loss derivative. See (\ref{eqnn9824}) for more details.

\subsection{Estimates of the Taylor coefficient} The main goal of this subsection is to obtain some basic estimates for the Taylor coefficient, which is a necessary part of the energy estimate.

\begin{lemma}\label{taylorcoefficient}
Under the smallness estimate \textup{(\ref{smallnessestimate})}, we have the following estimates for $\gamma\leq 3, \gamma_1\leq 2$,
\[
\|a-1 \|_{H^k}\lesssim \|\nabla\,h\|_{H^k} + \|(\,h, \nabla\psi)\|_{H^k}\|(\,h, \nabla\psi)\|_{\widetilde{W^2}},
\|a-1\|_{\widetilde{W^{\gamma}}}\lesssim \|\nabla\,h\|_{\widetilde{W^\gamma}} +  \big[ \| \nabla\psi\|_{\widetilde{W^{\gamma+1}}} + \|\nabla\,h\|_{\widetilde{W^{\gamma}}}  \big]^2,
\]
\[
\| \p_t a \|_{H^k}\lesssim \|\nabla \psi\|_{H^{k+1}} +\|(\,h, \nabla\psi)\|_{\widetilde{W^{3}}}  \|(\,h,\nabla\psi)\|_{H^{k+1}},
\|\p_t a \|_{\widetilde{W^{\gamma_1}}} \lesssim  \| \Lambda^2\psi \|_{\widetilde{W^{\gamma_1+1}}} +  \| (\nabla\,h, \nabla\psi)\|_{\widetilde{W^{\gamma_1+1}}}^2.
\]
\end{lemma}
\begin{proof}

Recall  (\ref{hvderivative}), (\ref{waterwave}), and
\[
a= 1 + V\cdot\nabla B + \p_t B,\quad B=\p_z \varphi/(1+\,h)\big|_{z=0}.
\]
 To estimate $a$ and $\p_t a$,  it's sufficient to estimate $\p_z\p_t\varphi$ and $\p_z\p_t^2\varphi $. From the fixed point type formulation of $\nabla_{x,z}\varphi$ in (\ref{fixedpoint}), we can derive the following equality, 
\[
\nabla_{x,z}\p_t \varphi = \Bigg[ \Big[ \frac{e^{-(z+1)\d}+ e^{(z+1)\d}}{e^{-\d} + e^{\d}}\Big]\nabla\p_t \psi ,  \frac{e^{(z+1)\d}- e^{-(z+1)\d}}{e^{-\d}+e^{\d}} \d \p_t\psi\Bigg] + [\mathbf{0}, \p_t g_1(z)]+ 
\]
\[
+\int_{-1}^{0} [K_1(z,s)-K_2(z,s)-K_3(z,s)](\p_t g_2(s)+\nabla \cdot \p_t g_3(s))  ds \]
\begin{equation}\label{eqn3300}
+\int_{-1}^{0} K_3(z,s)\d\textup{sign($z-s$)}\p_t g_1(s)  -\d [K_1(z,s) +K_2(z,s)]\p_t g_1(s)\, d  s.
\end{equation}
Following the same fixed point type argument that we used in the proof of  Lemma \ref{Sobolevestimate}, we can derive the following estimates,
\[
\| \nabla_{x,z}\p_t \varphi\|_{L^\infty_z \widetilde{W^\gamma}} \lesssim  \|\nabla\,h\|_{\widetilde{W^\gamma}} +   \| \p_t \,h\|_{\widetilde{W^{\gamma+1}}} \|\nabla_{x,z}\varphi\|_{L^\infty_z \widetilde{W^\gamma}}
\]
\begin{equation}\label{eqn3301}
 \lesssim
\|\nabla\,h\|_{\widetilde{W^\gamma}} +  \big[ \| \nabla\psi\|_{\widetilde{W^{\gamma+1}}} + \|\nabla\,h\|_{\widetilde{W^{\gamma}}}  \big]^2.
\end{equation}
\[
\| \nabla_{x,z}\p_t \varphi\|_{L^\infty_z H^k}\lesssim \| \nabla \p_t \psi \|_{H^{k}} + \|\,h\|_{H^{k+1}} \| \nabla_{x,z}\p_t \varphi\|_{\widetilde{W^1}} +\|\p_t \,h\|_{H^{k+1}} \|\nabla_{x,z}\varphi\|_{\widetilde{W^{1}}}
\]
\[ + \|\p_t\,h\|_{\widetilde{W^1}} \|\nabla_{x,z}\varphi\|_{L^\infty_z H^{k}}\lesssim \|\nabla\,h\|_{H^{k}} + \|(\,h, \nabla\psi)\|_{H^k}\|(\,h, \nabla\psi)\|_{\widetilde{W^2}}.
\]

We can take another time derivative at the both hands side of equation (\ref{eqn3300}) to derive a fixed point type formulation for $\nabla_{x,z}\p_t^2 \varphi$. Following a similar argument, we can derive the following estimate, 
\[
\|\nabla_{x,z}\p_t^2 \varphi\|_{L^\infty_z \widetilde{W^\gamma}} \lesssim \|\p_t^2\psi \|_{\widetilde{W^{\gamma +1}}} + \|\p_t \,h\|_{\widetilde{W^{\gamma +1}}} \| \p_t \nabla_{x,z}\varphi\|_{L^\infty_z \widetilde{W^\gamma} } + \|\p_t^2\,h \|_{\widetilde{W^{\gamma +1}}} \| \nabla_{x,z}\varphi\|_{L^\infty_z \widetilde{W^\gamma}}.
\]
Recall the system of equations satisfied by $\,h$ and $\psi$ in (\ref{waterwave}), we have
\[
\p_t^2 h = \p_t G(h)\psi = \p_t[(1+|\nabla h|^2) B -\nabla h \cdot \nabla \psi],
\]
\[
\p_t^2 \psi = -\p_t \,h + \nabla\psi \cdot\nabla\p_t \psi + (1+|\nabla\,h|^2) B \p_t B + \nabla\,h\cdot \nabla\p_t \,h B^2.
\]
Hence,
\[
\|\p_t^2\psi\|_{\widetilde{W^{\gamma+1}}} \lesssim \| \Lambda^2\psi \|_{\widetilde{W^{\gamma+1}}} + \big[ \| \nabla\,h\|_{\widetilde{W^{\gamma+1}}} + \|\nabla\psi\|_{\widetilde{W^{\gamma+1}}}\big]^2.
\]
Combining above estimate, (\ref{eqn3301}) and (\ref{eqn2201}) in Lemma \ref{Sobolevestimate}, we have
\[
\|\nabla_{x,z}\p_t^2 \varphi\|_{L^\infty_z \widetilde{W^\gamma}} \lesssim   \| \Lambda^2\psi \|_{\widetilde{W^{\gamma+1}}} + \big[ \| \nabla\,h\|_{\widetilde{W^{\gamma+1}}} + \|\nabla\psi\|_{\widetilde{W^{\gamma+1}}}\big]^2.
\]
Follow the same argument, we derive the $L^2$ type estimate of $\p_t^2\varphi$ as follows,
\[
\| \nabla_{x,z}\p_t^2 \varphi\|_{L^\infty_z H^k} \lesssim \| \p_t^2\psi\|_{H^{k+1}} + \|\p_t \,h\|_{H^{k+1}}\|\nabla_{x,z}\p_t\varphi\|_{L^\infty_z \widetilde{W^1}} 
\]
\[
+ \|\p_t \,h\|_{\widetilde{W^1}} \| \nabla_{x,z}\p_t\varphi\|_{L^\infty_z H^{k}}+ \|\p_t^2\,h\|_{H^{k}} \|\nabla_{x,z}\varphi\|_{L^\infty_z \widetilde{W^1}} + \|\p_t^2\,h\|_{\widetilde{W^1}} \|\nabla_{x,z}\varphi\|_{L^\infty_z H^k} 
\]
\[
+\|\,h\|_{H^{k+1}} \|\p_t^2 \nabla_{x,z}\varphi\|_{L^\infty_z \widetilde{W^1}} +\|\,h\|_{\widetilde{W^1}} \| \nabla_{x,z}\p_t^2 \varphi\|_{L^\infty_z H^k},
\]
which further gives us the following estimate, 
\[
\| \nabla_{x,z}\p_t^2 \varphi\|_{L^\infty_z H^k} \lesssim \|\Lambda^2 \psi\|_{H^{k+1}} + \|\nabla\,h\|_{\widetilde{W^2}} \|\Lambda^2\psi\|_{H^{k+1}} + \|\nabla\psi \|_{\widetilde{W^2}} \|\,h\|_{H^{k+1}}
\]
\[
\lesssim \|\nabla \psi\|_{H^{k+1}} + \|(\,h, \nabla\psi)\|_{\widetilde{W^{3}}}  \|(\,h,\nabla\psi)\|_{H^{k+1}}.
\]
Therefore, our desired estimates of the Taylor coefficient hold.

 \end{proof}

\section{Energy estimate}\label{energyestimatesection}
The main goal in this section is to prove our main Theorem \ref{maintheorem}.   Since the energy of $(U_1,U_2)$ is comparable with  the energy of $(\,h, \Lambda \psi)$ (see(\ref{eqn5900})), it is sufficient to do energy estimate for the system of equations satisfied by $U_1$ and $U_2$. Let $N_0$ be the prescribed top regularity level.  From (\ref{equation200}) and (\ref{equation201}), we know that  the system of equations satisfied by $(U_1,U_2)$ reads as follows, 
\begin{equation}\label{system4}
\left\{\begin{array}{l}
\p_t U_1 -\Lambda U_2 = T_{\sqrt{\lambda}\alpha} U_2 - T_{V\cdot}\nabla U_1 +\mathcal{R}_1, \\
\\
\p_t U_2 + \Lambda U_1 = -T_{\sqrt{\lambda}\alpha} U_1 - T_{V}\cdot \nabla U_2 + \mathcal{R}_2. \ \\
\end{array}\right.
\end{equation}
The precise formulations of good remainder terms $\mathcal{R}_1$ and $\mathcal{R}_2$ are not so important in the energy estimate part. From (\ref{equation200}) and (\ref{equation201}), we know that they are good error terms, i.e.,
\begin{equation}\label{remainderenergy}
\| \mathcal{R}_1\|_{H^{N_0}} + \| \mathcal{R}_2\|_{H^{N_0}} \lesssim_{N_0} \big[ \|(\,h, \Lambda\psi)\|_{\widehat{W^{4,\alpha}}}  + \|(\,h, \Lambda\psi)\|_{\widehat{W^4}}^2\big] \|(\eta,\Lambda \psi)\|_{H^{N_0}}.
\end{equation}
Define the energy of $U_1$ and $U_2$ as follows,
\begin{equation}\label{energyfunction}
E_{N_0}(t):=\h\Big[ \|U_1\|_{L^2}+\|U_2\|_{L^2} + \sum_{k+j=N_0, 0\leq k , j\in \mathbb{Z}} \| \p_1^{k}\p_2^{j} U_1\|_{L^2}^2 + \|\p_1^{k}\p_2^{j}  U_2\|_{L^2}^2\Big].
\end{equation}
From 
(\ref{system4}), we have
\[
\Big|\frac{d}{d t} E_{N_0}(t) \Big|\lesssim \| (U_1, U_2)\|_{H^{N_0}}\| (\mathcal{R}_1, \mathcal{R}_2)\|_{H^{N_0}}+ \sum_{k+j=N_0, 0\leq k , j\in \mathbb{Z}}  \Big|\int_{\R^2}\big[\p_1^{k}\p_2^{j}  U_1\p_1^{k}\p_2^{j} [-T_{V}\cdot \nabla U_1]  
\]
\[
+ \p_1^{k}\p_2^{j}  U_2\p_1^{k}\p_2^{j} [-T_{V}\cdot \nabla U_2]\big] d x \Big|+ \Big| \int_{\R^2} \p_1^{k}\p_2^{j}  U_1 \p_1^{k}\p_2^{j} ([ T_{\sqrt{\lambda}\alpha} U_2]) -\p_1^{k}\p_2^{j}  U_2 \p_1^{k}\p_2^{j} ([T_{\sqrt{\lambda}\alpha} U_1])\Big|
\]
\[
\lesssim_{N_0} \big[ \| \nabla V\| _{\widetilde{W}^{1}} + \|(\,h, \Lambda\psi)\|_{\widehat{W^{4,\alpha}}}  + \|(\,h, \Lambda\psi)\|_{\widehat{W^4}}^2 \big]\|(U_1, U_2)\|_{H^{N_0}}^2 + \mathcal{E}_{N_0},
\]
\begin{equation}\label{eqn811}
\lesssim_{N_0}  \big[  \|(\,h, \Lambda\psi)\|_{\widehat{W^{4, \alpha}}} + \| (\,h,\Lambda\psi)\|_{\widehat{W^{4}}}^2\big]\|(U_1, U_2)\|_{H^{N_0}}^2 + \mathcal{E}_{N_0},
\end{equation}
where
\[
\mathcal{E}_{N_0} =  \sum_{k+j=N_0, 0\leq k , j\in \mathbb{Z}}\Big|\int_{\R^2}\Big[\p_1^{k}\p_2^{j} U_1  T_{\sqrt{\lambda}\alpha} \p_1^{k}\p_2^{j} U_2 -\p_1^{k}\p_2^{j} U_2 T_{\sqrt{\lambda}\alpha}  \p_1^{k}\p_2^{j}  U_1 \Big]\Big|\]
\be\label{eqnn9824}
= \sum_{k+j=N_0, 0\leq k , j\in \mathbb{Z}}\Big|\int_{\R^2} \Big[ \p_1^{k}\p_2^{j} U_1 \big(T_{\sqrt{\lambda}\alpha}  - (T_{\sqrt{\lambda}\alpha})^{\ast}\big)\p_1^{k}\p_2^{j} U_2\Big]\Big|.
\ee
Recall that $\sqrt{\lambda}\alpha\in M^{1/2}_1$ is a symbol of order $1/2$. Note that it is real. Hence, from Lemma \ref{adjoint}, we know that the symbol of the adjoint operator is itself, and operator $(T_{\sqrt{\lambda}\alpha})^{\ast} - T_{\sqrt{\lambda} \alpha}$  is of order $-1/2$. As a result,    the following estimate holds, 
\begin{equation}\label{eqn808}
\mathcal{E}_{N_0}\lesssim M_{1}^{1/2}(\sqrt{\lambda}\alpha) \| (U_1, U_2)\|_{H^{N_0}}^2 \lesssim \| \nabla \,h\|_{\widetilde{W^2}}\|(U_1, U_2)\|_{H^{N_0}}^2.
\end{equation}
Combining above estimate with (\ref{eqn811}) and (\ref{eqn5900}), we have
\[
\Big| \frac{d}{d t} E_{N_0}(t) \Big| \lesssim_{N_0}   \big[  \|(\,h, \Lambda\psi)\|_{\widehat{W^{4, \alpha}}} + \| (\,h,\Lambda\psi)\|_{\widehat{W^{4}}}^2\big]\|(h, \Lambda \psi)\|_{H^{N_0}}^2.
\]

\section*{Appendix: Quadratic terms of the good remainders}\label{appendix}

In this section, we calculate explicitly the quadratics terms of the good reminder terms $\mathcal{R}_1$ and $\mathcal{R}_2$ 	to help readers to understand the fact that we can gain  one derivative  in the new energy estimate (\ref{energyestimate}) for the inputs of quadratic terms, which are putted  in the $L^\infty$ type space.     Recall (\ref{transform}) and (\ref{taylorcoef}). We have 
 \[
\Lambda_1[B]= \Lambda^2 \psi, \quad \Lambda_{1}[a]= \Lambda_{1}[\p_t B]= -\Lambda^2 h,\quad \Lambda_{1}[\alpha]= -\h \Lambda^2 h,\quad \Lambda_{1}[\beta]=0.
\]
Recall (\ref{system4}). By using above definitions, we can reduce the equations satisfied by $U_1$ and $U_2$ to the equations satisfied by $h$ and $\psi$ as follows,  
\begin{equation}\label{eqn2}
\left\{\begin{array}{l}
\p_t\,h -\Lambda^2\psi = \widetilde{Q}_1(\,h, \psi) + \Lambda_{2}[\mathcal{R}_1](\,h, \psi) + \textup{cubic and higher}_1, \\
\\
\p_t \Lambda \psi + \Lambda \,h = \widetilde{Q}_{2}(\,h, \psi) + \Lambda_{2}[\mathcal{R}_2](\,h, \psi) + \textup{cubic and higher}_2,\\
\end{array}\right.
\end{equation}
where
\[
\widetilde{Q}_1(\,h, \psi)=-\Lambda^2(T_{\Lambda^2\psi}\,h) + \h T_{\Lambda^2\,h}\Lambda^2\psi + \h T_{\Lambda^4\psi}\,h - \h(T_{\Lambda^2\,h} |\nabla|^{1/2}\Lambda \psi) - T_{\nabla \psi} \cdot \nabla \,h,
\]
\[
\widetilde{Q}_2(\,h, \psi)= \Lambda(T_{\Lambda^2\psi} \Lambda^2\psi-T_{\Lambda^2\,h}\,h) +\Lambda (\h T_{\Lambda^2\,h}\,h) +\h ( T_{\Lambda^2\,h}\d^{1/2}\,h) - T_{\nabla\psi}\cdot \nabla \Lambda \psi.
\]
Recall   (\ref{waterwave}) and (\ref{eqq345}) in Lemma \ref{quadraticterms}. We   have
\[
\Lambda_2[\mathcal{R}_1](h, \psi) = \Lambda_{2}[G(h)\psi]- \widetilde{Q}_{1}(h, \psi)= -\nabla\cdot\big( T_{ h}\nabla \psi \big)- \nabla\cdot\mathcal{R}( h, \nabla \psi) -T_{\Delta^2 \psi} h -\Lambda^{2}(T_{ h } \Lambda^{2}\psi )  
\]
\be\label{eqqqn1234}
- \Lambda^{2}\mathcal{R}(h, \Lambda^2 \psi)+ \frac{1}{2} T_{\Lambda^2 h} \Lambda (\Lambda-\d^{1/2})\psi -  \h T_{\Lambda^4\psi}\,h,
\ee
\[
\Lambda_{2}[\mathcal{R}_2](h, \psi) = \Lambda\big[ -\frac{1}{2} |\nabla\psi |^2 + \h |\Lambda^2\psi|^2 \big]- \widetilde{Q}_2(h, \psi) = \Lambda(-T_{\nabla \psi}\cdot \nabla\psi) + T_{\nabla\psi} \cdot \nabla\Lambda  \psi 
\]
\be\label{eqnnqn1223}
+ \frac{1}{2}\big(  -\Lambda (T_{\Lambda^2 h }  h)+ T_{\Lambda^2 h} \d^{1/2} h    \big)  -\h\Lambda\mathcal{R}(\nabla \psi, \nabla \psi) + \h \Lambda\mathcal{R}(\Lambda^2 \psi, \Lambda^2 \psi).
\ee
Note that 
\be\label{difference}
\Lambda-\d^{1/2}= \d^{1/2}( \sqrt{\tanh\d}-1)= \frac{-2 e^{-\d}\d^{1/2}}{(\sqrt{\tanh\d}+1)(e^{\d}+e^{-\d})}.
\ee
Now, it is easy to see that  $\Lambda_{2}[\mathcal{R}_2](h, \psi)$ and  $\Lambda_{2}[\mathcal{R}_2](h, \psi)$ don't lose derivatives. It remains to check  that we can gain one derivative in the $L^\infty$ type space. From (\ref{eqqqn1234}) and (\ref{eqnnqn1223}), it would be sufficient to check the following  term ,
\be\label{bulkterms}
-\nabla\cdot \big(T_{ h }  \nabla \psi \big)  -\Lambda^{2}(T_{ h } \Lambda^{2}\psi ).
\ee
The corresponding symbol of above quadratic terms is given as follows,
\[
\xi\cdot \eta -|\xi||\eta|\tanh |\xi| \tanh|\eta|,\quad |\xi-\eta|\ll |\xi|\sim |\eta|.
\]
We decompose this symbol  into two parts as follows, 
\[
p_1(\xi-\eta, \eta)=\xi\cdot \eta - |\xi||\eta|=-\h |\xi-\eta|^2  + \h |\xi|^2+ \h|\eta|^2-|\xi||\eta|=-\h |\xi-\eta|^2 + \h(|\xi|-|\eta|)^2,
\]
\[
  p_2(\xi-\eta, \eta) = |\xi||\eta| \big(1-\tanh|\xi|\tanh|\eta| \big).
\]
Now, it is clear that the first part of (\ref{bulkterms}), which is determined by $p_1(\xi-\eta, \eta)$, doesn't lose derivative and gain two derivatives for `` $h$''. For the second part of (\ref{bulkterms}), which is determined by $p_2(\xi-\eta, \eta)$, we can lower the regularity to $L^2$ such that we can put ``$  \psi$''   in $L^\infty$ and`` $h$'' in $L^2$. As a result, we can always gain one derivative for the inputs of quadratic terms that are putted in $L^\infty$ type space.

\end{document}